\pdfoutput=1
\documentclass[twoside, 12pt]{article}
\usepackage[utf8]{inputenc}
\usepackage[T1]{fontenc}
\usepackage[charter]{mathdesign}
\usepackage{amsthm, amsfonts, enumerate, calligra,mathrsfs}
\usepackage[a4paper,margin=3.3cm, headsep=15pt, headheight=2cm]{geometry}
\newcommand\smvee{\raise0.3ex\hbox{$\scriptscriptstyle\vee$}}
\usepackage{fancyhdr}
\usepackage{tikz-cd, tikz}
\usepackage[all,cmtip]{xy}
\usetikzlibrary{matrix, calc}
\usepackage{adjustbox,lipsum}
\usepackage{mathtools}

\newtheorem{theorem}{Theorem}[section]
\newtheorem{nmtheorem}{Theorem}

\newtheorem{prop}[theorem]{Proposition}
\newtheorem{remark}[theorem]{Remark}

\newtheorem{defn}[theorem]{Definition}

\numberwithin{equation}{section}
\newcounter{para}

\usepackage[colorlinks]{hyperref}
\usepackage[nameinlink]{cleveref}
\hypersetup{colorlinks,breaklinks,
	citecolor=[rgb]{0.1,0.7,0.5},
	linkcolor=[rgb]{0.3,0.4,0.7}}
\newcommand{\etale}{étale }

\def\C{\mathbb{C}}
\def\Q{\mathbf{Q}}
\def\Qb{\mathbb{Q}}
\def\Ql{\mathbb{Q}_{\ell}}
\def\A{\mathbb{A}}

\def\P{\mathbb{P}}
\def\R{\mathbb{R}}
\def\PSig{\mathbb{P}({\Sigma})}

\def\N{\mathbb{N}}

\def\Tc{\mathcal{T}}
\def\Ic{\mathcal{I}}

\def\1{\mathbb{1}}

\def\cmord{\mathrm{Mor}_{\vec{d}}(C,\PSig)}
\def\pmord{\mathrm{Mor}_{\vec{d}}(\P^1,\PSig)}

\def\sgn{\mathit{sgn}}

\def\pic{\mathrm{Pic}}
\def\deg{\mathit{deg\,}}

\newcommand\Sym{\mathit{Sym}}

\date{}
\DeclareMathOperator{\Shom}{\mathscr{H}\text{\kern -3pt {\calligra\large om}}\,}
\DeclareMathOperator{\Sext}{\mathscr{E}\text{\kern -3pt {\calligra\large xt}}\,}
\newcount \questionno
\def\question{\medbreak
	\global \advance \questionno 1
	\textbf{Problem \the\questionno}.\enspace \ignorespaces}
\newcommand\shorttitle{Curves on toric varieties}
\newcommand\authors{Oishee Banerjee}
\usepackage{titlesec}

\titleformat{\section}
{\normalfont\fontsize{11}{13}\bfseries}{\thesection}{1em}{}

\title{Moduli of curves on toric varieties\\and their stable cohomology}
\author{Oishee Banerjee}

\begin{document}
	
	\maketitle
	\begin{abstract}
		We prove that the  cohomology of the moduli space of morphisms of a fixed finite degree from a smooth projective curve $C$ of genus $g$ to a complete simplicial toric variety $\PSig$, denoted by the rational polyhedral fan $\Sigma$, stabilizes. As an arithmetic consequence we obtain a resolution of the Batyrev-Manin conjecture for toric varieties over global function fields in all but finitely many characteristics.
	\end{abstract}

\section{Introduction}

Noting that the moduli space of degree $d$ morphisms from a smooth projective curve of genus $g$ to the projective space $\P^n$ (or more generally, to weighted porjective stacks, see \cite{BPS} and the references therein) is well studied so far as (\etale) homological stability is concerned, the next level of generalization is to ask the same question for toric varieties- a question whose arithmetic counterpart is the Batyrev-Manin type conjecture over global function fields.

 To elaborate, fix an algebraically closed field $K$. Let $\Sigma$ be a complete, simplicial rational polyhedral fan in $N\otimes \R\cong \R^r$, $N$ a lattice, and let $M$ be the dual lattice of $N$, their respective basis given by $e_1,\ldots, e_r$, and its dual $e^1,\ldots, e^r$. Let $\mathbb{P}_{\Sigma}$ be the corresponding complete simplicial toric variety, and $\Sigma(1)$ denote the set of rays, $\Sigma(1):=\{\rho_1,\ldots, \rho_n\}$, and $v_{1},\ldots, v_n$ the integer generators of those rays. We denote cones in $\Sigma$ by $\sigma$, the set of all rays in $\sigma$ by $\sigma(1)$ and let $\Sigma_{\mathrm{max}}$ denote the set of maximal cones in $\Sigma$.
A morphism from a smooth projective $C$ of genus $g$ to $\PSig$ corresponds to what we call a \emph{non-degenerate $\Sigma$-collection} which consists of the following data:
\begin{enumerate}[i]
	\item line bundles $L_{i}$ for each $\rho_i \in \Sigma(1)$,
	\item global sections $s_{i}\in H^0(C,L_{i})$ for each $\rho_i \in \Sigma(1)$,
	\item \emph{compatibility isomorphisms:} $\alpha_{e^i}: \bigotimes_{\rho} L^{\langle e^i, v_{j}\rangle} \to \mathcal{O}_C$,
	\item satisfying the \emph{nondegeneracy} condition: suppose $s_{j}^*: L_{j}^{-1} \to \mathcal{O}_C$ denotes the map induced by the section $s_{j}$ for each $\rho_j \in \Sigma(1)$, then there is a surjection of locally free sheaves $$ \bigoplus_{\sigma\in \Sigma_{\mathrm{max}}} \big(\otimes_{\rho\notin \sigma(1)} s^*_{\rho}\big): \bigoplus_{\sigma \in \Sigma_{\mathrm{max}}} \big(\otimes_{\rho\notin \sigma(1)} L_{\rho}^{-1}\big) \to \mathcal{O}_C.$$
	If the last condition of non-degeneracy is not satisfied we simply call it a \emph{$\Sigma$-collection}.
\end{enumerate}
\begin{remark}
	Note our distinction in terminology from what Cox uses- $\Sigma$-collections in \cite{Coxfunctor} are, by definition, non-degenerate. 
\end{remark}In the language of Cox (see \cite{Coxfunctor}) the collection of triples $(L_{\rho}, s_{\rho}, \alpha_u)_{\rho\in \Sigma(1), u\in M}$ satisfying the above conditions is a \emph{$\Sigma$-collection} on $C$ (note that the definition of $\Sigma$-collections holds for any $K$-scheme $X$, not just for a smooth projective curve). 

Let $\vec{d}:= (d_1, \ldots, d_n)\in \N^n$ be a \emph{degree vector}, where $n:= |\Sigma(1)|$ and let $\mathrm{Mor}_{\vec{d}}(C,\PSig)$ denote the moduli space of degree $\vec{d}$ maps from $C$ to $\PSig$ given by $\Sigma$-collections on $C$ such that degree of $L_i = d_i$ (note that by condition iii above, this forces $\sum_i d_i u_i=0$ for the space of morphisms to be non-empty). A \emph{primitive collection} (see \cite{CLS}) of rays in $\Sigma$ is a subset of $\Sigma(1)$ which does not span a cone in $\Sigma$, but every subset of that collection does. $E_1,\ldots, E_t$ be its primitive collections. We now state our main theorem.
\begin{nmtheorem}\label{thm:CPSig}
	Let $C$ be a smooth projective curve of genus $g$, $\PSig$ a complete simplicial  toric variety defined by a rational polyhedral fan $\Sigma$, let $\{\rho_1,\ldots, \rho_n\}$ be its rays, and let $\vec{d}:= (d_1, \ldots, d_n)\in \N^n$ be a degree vector such that $d_i\geq 2g$ for all $i$ and $\sum d_iu_i= 0$.  Let $n_0:=\min\{d_1,\ldots, d_n\}-2g$. Then there exists a second quadrant spectral sequence, which converges to $H^*(\cmord; \Qb)$ an algebra, which has the following description. The $E_2$ term is a bigraded algebra that collapses on $E_2^{-p,q}\Big\vert_{p\leq n_0}$. Furthermore, $E_2^{-p,q}\Big\vert_{p\leq n_0}$ is a quotient of the graded commutative $\Qb$-algebra
$$H^*(J(C);\Qb)^{\otimes (n-r)}\otimes  \Qb[D_1,\ldots, D_n]/ I  \otimes \big(\oplus_{j=1}^{t}\wedge \Qb\{z_j\}\big)\otimes \big(\oplus_{j=1}^{t}\Sym \Qb\{\alpha^j_1,\ldots, \alpha^j_{2g}\}\big),$$ 
where $t$ is the number of primitive collections of $\Sigma$, $H^i(J(C);\Qb)$ has degree $(0,i)$, $D_i$ has degree $(0,2)$ for all $i$, $z_j$ has degree $(-1,2\epsilon_j)$, $\alpha^j_i$ has degree $(-1,2\epsilon_j-1)$ for all $i$, modulo elements of degree $(-i,j)$ with $i>n_0$ and the ideal $I$ is generated by all \begin{enumerate}
	\item $D_{i_1}\cdot\ldots\cdot D_{i_k}$ for $v_{i_1},\ldots, v_{i_k}$ not in a cone of $\Sigma$;
	\item $\sum_{i=1}^{r} \langle v,u_i\rangle$ for $v\in M$;
	\item $\sum_{j=1}^{k} D_{i_1}\ldots \widehat{D_{i_j}}\ldots D_{i_k}$ for every primitive collection $\{u_{i_1}, \ldots, u_{i_k}\}\subset \Sigma(1)$, where $\widehat{D_{i_j}}$ denotes the $j^{th}$ entry removed.
\end{enumerate} 
\noindent  Furthermore this is a spectral sequence of mixed Hodge structures, with $\Qb\{\alpha^j_1,\ldots, \alpha^j_{2g}\}$ carrying a pure Hodge structure of weight $2\epsilon_j-1$ and $\Qb\{z_j\}$ carrying a pure Hodge structure of weight $2\epsilon_j$ for all $1\leq j\leq t$, and $D_i$ is of type $(1,1)$ for all $i$. 
\end{nmtheorem}

In the case when $C=\P^1$ the result above simplifies exponentially, in a sense, thanks to the class group of $\P^1$ being trivial.
\begin{nmtheorem}\label{thm:PPSig}
	Let $\P_{\Sigma}$ be a complete simplicial toric variety defined by the rational polyhedral fan $\Sigma$ and let $\{\rho_1,\ldots, \rho_n\}$ be its rays. Let $\vec{d}$ be a degree vector satisfying $\sum d_iu_i= 0.$ Then $$\pmord \cong \Qb[D_1,\ldots, D_n]/ I \otimes \big(\oplus_{1\leq j\leq t}\wedge\Qb\{z_j\} \big)$$ where $t$ is the number of primitive collections of $\Sigma$, $z_j$ has cohomological degree $2\epsilon_j-1$ and of type $(\epsilon_j, \epsilon_j)$, $D_i$ of cohomological degree $2$ and type $(1,1)$ and $I$ is the ideal generated by all \begin{enumerate}
		\item $D_{i_1}\cdot\ldots\cdot D_{i_k}$ for $v_{i_1},\ldots, v_{i_k}$ not in a cone of $\Sigma$;
		\item $\sum_{i=1}^{r} \langle v,u_i\rangle$ for $v\in M$;
		\item  $\sum_{j=1}^{k} D_{i_1}\ldots \widehat{D_{i_j}}\ldots D_{i_k}$ for every primitive collection $\{\rho_{i_1}, \ldots, \rho_{i_k}\}\subset \Sigma(1)$, where $\widehat{D_{i_j}}$ denotes the $j^{th}$ entry removed.
	\end{enumerate}
\end{nmtheorem}
An immediate consequence of applying the Grothendieck-Lefschtez fixed point theorem (see \cite{Grothendieck}) to the cohomology results above is that the $\mathbb{F}_q$-points of these moduli spaces is of the order of $q^{\sum_{i=1}^{n}d_i-ng+r}$, where $\mathrm{dim}\,\,\, \cmord = \sum_{i=1}^{n}d_i-ng+r$.

\paragraph{Some context.}\begin{enumerate}
	\item \textit{On toric varieties.} The moduli space of curves in toric varieties, under various constraints has been an object of interest from the angles of tropical geometry and enumerative geometry for a while now. The moduli space of \emph{stable} maps naturally occur in enumerative geometry and mirror symmetry. However this note is entirely unrelated to those. Our focus is not stable, but actual algebraic maps from smooth projective curves to complete simplicial toric varieties- their topology and arithmetic, and thus working towards an entirely different goal than enumerative geometry.
	
	\item \textit{On morphism-spaces and Batyrev-Manin type conjectures.} Morphism-spaces have been studied in various guises from various angles. 
	\begin{itemize} 
		\item In topology, there is an influential paper of Segal (\cite{Segal})where he worked specifically with projective spaces, using scanning maps to compute the stable homology of what he calls the moduli space of `based' holomorphic maps $\mathrm{Rat}^*_n(\C\P^1,\C\P^r)$, the moduli space of $(r+1)$ polynomials in one variable of degree $n$, up to $\C^{\times}$. He generalizes it to higher genera in the domain. To this date, to the best of our knowledge, all generalizations (e.g replacing $\P^1$ by $\P^n$ with $n\geq 1$, or replacing $\C\P^r$ by a Grassmanian) all critically uses the key idea of Segal's. From the point of view topology, the last known activity seems to be Guest's beautiful paper (see \cite{Guest}) where he shows that the moduli space of holomorphic maps of a fixed degree from $\P^1$ to $\PSig$ (which, he assumed is smooth and projective), has relations with configuration spaces of $C$ with labels in a partial monoid, and uses this connection to show that this moduli space is homotopy equivalent, up to a certain range (which he computed in the case of some specific examples of $\PSig$) i.e. stably. 
		\item In number theory, in a landmark paper by Browning and Sawin (see\cite{BS}) they handle the case when the target projective space is replaced by a hypersurface of low enough degree by using a geometric version of the Hardly-Littlewood circle method. 
	\end{itemize}
In a sense lying in the middle, our methods are, on one hand, homotopy theoretic resulting in the understanding of the topology of $ \cmord$, and on the other hand all maps are algebraic, thus keeping track of the arithmetic data every step of the way.
In turn we see a much stronger version of Guest's with techniques that bypass Segal's scanning maps. To be precise, our computation in the case of $\P^1$ is the complete cohomology ring (i.e. not just the stable part, unlike Guest's) with Mixed Hodge structure, and for higher genus, a spectral sequence which gives the stable cohomology when the degree is large enough. 

There's one slight caveat-- unlike Guest, we work with cohomology with rational coefficients (whereas his are with integer coefficients). Although, theoretically, our methods can be moulded to work for cohomology with $\mathbb{Z}$-coefficients in practice the computations get unnecessarily difficult, with no impact to the arithmetic implications.
\end{enumerate}

\paragraph{Plan of attack.}
Keeping aside technical algebro-geometric difficulties stemming from genus considerations, the method for both the theorems is essentially homotopy theoretic. It makes crucial use of \cite{Banerjee21}, or more accurately the proof of it, and is, in some ways, naturally similar to the case when the  toric variety in question is replaced by a weighted projective stack (whose generic point has a non-trivial stabilizer)- a problem answered by Banerjee, Park and Schmitt in \cite{BPS}. This also brings to the forefront the fact that our method behaves well in the stacky world because the key lemma in our proof (\cite[Lemma 2.11]{Banerjee21}) is a statement that holds with sheaves of $\Qb$-vector spaces as well as $\ell$-adic sheaves in \etale topology. To portray the general philosophy behind these results we first prove the case of $C=\P^1$ and then move on to the higher genus case.

\section{Cox ring and its generalization to toric bundles.}
In this section we briefly recall Cox's method (see \cite{Coxring}) of attaching a homogenous polynomial ring to a toric variety $\PSig$ (which need not be complete, or simplicial for most definitions to work, but we still assume them for our purposes), a monomial ideal, and a reductive linear algebraic group, such that $\PSig$ can be constructed as a categorical quotient (which is a geometric quotient in case the toric variety is complete and simplicial), much like we attach a homogenous coordinate ring, the irrelevant ideal to a projective space and the $1$-dimensional torus $\mathbb{G}_m$.

Then we use it in a relative setting, defining what is called a toric bundle, and we compute its cohomology (which we use in the proof of Theorem \ref{thm:CPSig}.)

Very briefly we recall parts of \cite{Coxring}. Let $\Sigma$ be a rational polyhedral fan, and let $\PSig$ be the corresponding toric variety. Let $\Sigma(1) =\{\rho_1,\ldots, \rho_n\}$ be the set of rays. The \emph{homogenous coordinate ring of $\PSig$} is defined by $$ K[x_1,\ldots, x_n]$$ (with grading given by the class group $\mathrm{Cl}(\Sigma)$ of $\PSig$, which we will need later). The \emph{exceptional locus} $Z(\Sigma)$ is defined as the space in $\A^n=\mathrm{Spec} (K[x_1,\ldots, x_n])$ cut out by the monomial ideal $B(\Sigma)\subset  K[x_1,\ldots, x_n]$ which is generated by monomials of the form $$x_{i_1}\cdot\ldots x_{i_{\epsilon_k}}$$ for each \emph{primitive collection} $E_k=\{\rho_{i_1},\ldots, \rho_{\epsilon_k}\}$ of rays. Let $G(\Sigma):= \mathrm{Hom}(\mathrm{Cl}(\Sigma),\mathbb{G}_m)$ act on $\A^n-Z(\Sigma)$ by $$g\cdot (t_1,\ldots, t_n):= (g(D_1) t_1, \ldots, g(D_n) t_n)$$ and the resulting geometric quotient (see \cite{CLS} or \cite[Theorem 2.1]{Coxring}) is $\PSig$. Mimicking the case of projective spaces, we will denote a point in $\PSig$ by $[t_1:\ldots:t_n]$  when there is no confusion, which denotes an equivalence class of a point $(t_1,\ldots,t_n)\in  \A^n-Z(\Sigma)$ under $G(\Sigma)$.

One can, of course, make a similar construction in a relative setting i.e. over a base scheme $X$. Let $\mathcal{F}$ be a vector bundle on $X$ that splits into sub-vector bundles  $\mathcal{F}_1\oplus\ldots\oplus \mathcal{F}_n$. The action of the reductive algebraic group $G(\Sigma)$ on $\mathcal{F}$ is locally given by $$g\cdot (s_1,\ldots, s_n):= (g(D_1)s_1, \ldots, g(D_n) s_n)$$ where for all sections $s_i$ of $\mathcal{F}_i$, $1\leq i\leq n$. The exceptional locus $\mathcal{Z}_X(\Sigma)$ is constructed likewise and, by abuse of notation, denoting the total space of $\mathcal{F}$ by $\mathcal{F}$, we call the geometric quotient $$\big(\mathcal{F}-\mathcal{Z}_X(\Sigma)\big)/ G(\Sigma)$$ a \emph{toric bundle}, with fibres naturally isomorphic to $\PSig$, and denote it by $\mathbb{P}_X(\mathcal{F},\Sigma)$. When $X$ is a point and in turn, $\mathcal{F}$ an affine space, we simply write (as we already did) $\PSig$.

The cohomology of $H^*(\PSig)$ is well known (see \cite{CLS} or \cite{Fulton}):
\begin{prop}\label{prop:CohomPSig}
	For a complete simplicial toric variety $\PSig$,  $$H^*({\PSig};\Qb) \cong \Qb[D_1,\ldots, D_n]/ I$$ where $D_{i}$ has cohomological degree $2$ and type $(1,1)$ and $I$ is the ideal generated by all \begin{enumerate}
		\item $D_{i_1}\cdot\ldots\cdot D_{i_k}$ for $\{\rho_{i_1}, \ldots, \rho_{i_k}\}$ a primitive collection in $\Sigma(1)$;
		\item $\sum_{i=1}^{r} \langle v,u_i\rangle  D_i$ for $v\in M$.
	\end{enumerate} 
\end{prop}

For any character $\lambda\in M$ there exists a $1$-dimension representation of the torus, and thus, there is an associated line bundle on $X$ which we denote by $\xi_{\lambda}$. In \cite{SU} Sankaran and Uma compute the cohomology (and Chow ring) of such toric bundles which we record here for future purposes.
\begin{prop}[\cite{SU}, Theorem 1.2]
Let $\Sigma$ be a rational polyhedral fan with rays $\{\rho_1,\ldots,\rho_n\}$. Let $\PSig\hookrightarrow \mathbb{P}_X(\mathcal{F},\Sigma)\to X$ be a toric bundle as above. Then the rational cohomology $H^*(\mathbb{P}_X(\mathcal{F},\Sigma);\Qb)$ is isomorphic, as a $H^*(X;\Qb)$-algebra, to $$H^*(X;\Qb)[D_1,\ldots, D_n]/I(\mathbb{P}_X(\mathcal{F},\Sigma))$$where $I(\mathbb{P}_X(\mathcal{F},\Sigma))$ is the ideal generated by all \begin{enumerate}
	\item $D_{i_1}\cdot\ldots\cdot D_{i_k}$ for $\{\rho_{i_1}, \ldots, \rho_{i_k}\}$ a primitive collection in $\Sigma(1)$;
	\item $\sum_{i=1}^{r} \langle e^j,u_i\rangle  D_i-c_1(\xi_{e^j})$ where $c_1(\xi_{e^j})\in H^2(X;\Qb)$ denotes the first Chern character of $\xi_{e^j}$.
\end{enumerate} 
\end{prop}

\section{Rational curves on toric varieties}\label{sec2}

\begin{proof}[Proof of Theorem \ref{thm:PPSig}]
	\textsc{Step 1: }\textit{Constructing some auxiliary toric varieties.}
	
\noindent	Let $\Gamma_{d_i}:=\Gamma(\P^1, \mathcal{O}_{\P^1}(d_i)) \cong \A^{d_i+1}$ for all $i=1,\ldots, n$. We now construct a toric variety ${\PSig}^{\vec{d}} $ corresponding a rational polyhedral fan which we denote by $\Sigma^{\vec{d}}$, as a geometric quotient as follows. 
	First, the one rays of $\Sigma^{\vec{d}}$ we define to be 
	\begin{gather*}
\Sigma^{\vec{d}}(1): = \{\rho_{1,0}, \ldots, \rho_{1,d_1}, \\  \hspace{3mm}   \rho_{2,0}, \ldots, \rho_{2,d_2}, \\ \ldots,\\ \rho_{n,0}, \ldots, \rho_{n,d_n}\}.
	\end{gather*} 
If $\{\rho_{i_1}, \ldots, \rho_{i_l}\}$ is a primitive collection in $\Sigma(1)$ then 	\begin{gather*}
 \{\rho_{i_1,0}, \ldots, \rho_{i_1,d_{i_1}}, \\  \hspace{3mm}   \rho_{i_2,0}, \ldots, \rho_{i_2,d_{i_2}}, \\ \ldots,\\ \rho_{i_l,0}, \ldots, \rho_{i_l,d_{i_l}}\}
\end{gather*} is a primitive collection in $\Sigma{\vec{d}} (1)$, and furthermore, these are the only primitive collections. We define the exceptional set $Z(\Sigma^{\vec{d}})$ in $\Pi_i \Gamma_{d_i} \subset \A^{\sum (d_i+1)}$ as that defined by all the primitive collections of $\Sigma^{\vec{d}}(1)$. Note that this determines $\Sigma^{\vec{d}}$ uniquely because the primitive collections determine the maximal cones (each maximal cone is generated by the rays that are in the complement of each primitive collection).
%Now consider $\R^{\sum_{i=1}^{n} (d_i+1)}$ with the rays above along the axes in the positive directions; we now remove the cones spanned by $\{\rho_{i_1,j_1},\ldots, \rho_{i_l,j_l}\}$ if and only if the corresponding $\{\rho_{i_1}, \ldots, \rho_{i_l}\}$ is a primitive collection in $\Sigma(1)$.  

The reductive algebraic group $G(\Sigma):= \mathrm{Hom}({\mathrm{Pic}}(\PSig), \mathbb{G}_m)$ naturally acts on  \linebreak $\A^{\sum (d_i+1)} - Z(\Sigma^{\vec{d}})$ be virtue of its action on $\A^{n}-Z(\Sigma)$ (recall that $|\Sigma(1)|=n$) and the resulting geometric quotient is ${\PSig}^{\vec{d}}$, which is complete and simplicial.

Note for any $n$-vector $\vec{a}:= \{(a_1,\ldots, a_n)\} \in \N^n$ (not necessarily a degree vector i.e. not necessarily satisfying $\sum_i a_i u_i= 0$) one can construct a a complete simplicial toric variety ${\PSig}^{\vec{a}}$ which is not unique, in the sense that our construction does not define $\Sigma^{\vec{a}}$ uniquely, however they have isomorphic Stanley-Reisner rings and in turn, isomorphic rational cohomology rings. Also observe that ${\PSig}^{(1,\ldots,1)}=\PSig$, the variety we started with.

 We finish Step 1 by recording  the (rational) cohomology of $\P_{\Sigma^{\vec{a}}}$ for future purposes. Let $D_1,\ldots, D_n$ be the torus invariant divisors of $\PSig$ corresponding to the minimal lattice points $u_1,\ldots, u_n$ along the rays $\rho_1,\ldots, \rho_n\in \Sigma(1)$. By the construction above, the torus invariant divisors of $\P_{\Sigma^{\vec{a}}}$ are
 	\begin{gather*}
 	\{D_{1,0}, \ldots, D_{1,a_1}, \\  \hspace{3mm}   D_{2,0}, \ldots, D_{2,a_2}, \\ \ldots,\\ D_{n,0}, \ldots, D_{n,a_n}\}.
 \end{gather*} 

From this, and the construction of $\P_{\Sigma^{\vec{a}}}$ above, we arrive at the following:
 \begin{prop}\label{prop:CohomPSig:veca}
Let $\P_{\Sigma^{\vec{a}}}$ be as defined above. Then $$H^*(\P_{\Sigma^{\vec{a}}};\Qb) \cong \Qb[\{D_{i,j}: 1\leq i\leq n, 0\leq j\leq a_i\}]/ I(\Sigma^{\vec{a}}) $$ where $D_{i,j}$ has cohomological degree $2$ and type $(1,1)$ and $I(\Sigma^{\vec{a}})$ is the ideal generated by all \begin{enumerate}
	\item $\Pi_{0\leq j \leq a_{i_1}}D_{i_1,j}\cdot\ldots\cdot \Pi_{0\leq j \leq a_{i_k}}D_{i_k,j}$ for $\{\rho_{i_1}, \ldots, \rho_{i_k}\}$ a primitive collection in $\Sigma(1)$;
	\item $\sum_{i=1}^{r} \langle v,u_i\rangle D_{i,j}$ where $0\leq j\leq d_i$, for $v\in M$.
\end{enumerate} 
 \end{prop}
	As an aside one should note that the same result holds for the rational Chow ring (see \cite{CLS} or \cite{Fulton}).

	\vspace{4mm}
	
	\textsc{Step 2: }\textit{Constructing a semi-simplicial space with the auxiliary toric varieties.}
	
\noindent	A morphism from $\P^1$ to $\PSig$ is given as $n$-tuple (again, recall $n=\Sigma(1)$ by definition) of sections $(s_1,\ldots, s_n)$, $d_i\in \Gamma_{d_i}$, such that \begin{enumerate}
		\item $\sum_i d_i u_i=0$
		\item $ s_i([a:b]) \notin Z(\Sigma)$;
	\end{enumerate} furthermore two such $n$-tuples, $(s_1,\ldots, s_n)$ and $(s'_1,\ldots, s'_n)$ give the same morphism if there exists $g\in \mathrm{Mor}({\mathrm{Pic}}(\PSig))$ such that $s_i= g(D_i)s'_i$ for all $i=1,\ldots, n$ (see \cite[Theorem 2.1]{Coxfunctor}).

	It follows that $\pmord$ is open dense in  $\P_{\Sigma^{\vec{d}}}$ and the complement, which we dub as the \emph{resultant $0$ locus} (in congruence with the very special case when $\PSig=\P^n$), is defined by $$\Upsilon({\vec{d}}) = \{(s_1,\ldots, s_n): s_i\in \Gamma_{d_i},  s_i([a:b]) \in Z(\Sigma) \text{ for some } [a:b]\in \P^1\}.$$ We are now going to construct a (semi)simplicial space augmented over $\Upsilon(\vec{d})$ as follows.

	 Let $\{E_1, \ldots, E_t\}$ be the set of primitive collections in $\Sigma(1)$, with cardinalities $\epsilon_1, \ldots, \epsilon_t$. For each $1\leq i \leq t$ let the rays of $E_i$ be denoted by $\{\rho_{i_1},\ldots, \rho_{i_{\epsilon_t}}\}$ and for each $E_i$ define a vector $1_{E_i}\in \N^n$ which has $1$s at the $i_1,\ldots, i_{\epsilon_t}$ coordinates and $0$ everywhere else. Also let $x,y$ denote the homogenous coordinates on $\P^1$. Let \begin{gather}
			X_0:= \bigsqcup_{1\leq i\leq t} \P^1\times \P_{\Sigma^{\vec{d}-1_{E_i}}}
		\end{gather}
	We define a map $\pi_0: X_0\to \P_{\Sigma^{\vec{d}}}$ by defining it on each component as follows: 
	\begin{gather*}
\P^1\times \P_{\Sigma^{\vec{d}-1_{E_i}}}\to \P_{\Sigma^{\vec{d}}}\\ [a:b], (s'_1,\ldots, s'_n)\mapsto (s_1,\ldots, s_n) \\ \text{such that }\begin{cases*}
s_j=s'_ j(ay-bx) \text{ whenever } \rho_j\in E_i; \\s_j=s'_j \text{ otherwise.}
\end{cases*}
	\end{gather*}
We call this map \emph{adding a basepoint} (in line with the case when the target space is just a projective space). Similarly, for $p\geq 0$ we define $X_p= \bigsqcup_{1\leq i_0\ldots, i_p\leq t}  (\P^1)^{p+1}\times \P_{\Sigma^{\vec{d}-\sum_{k=0}^{p} 1_{E_{i_k}}}}$ with maps on each component, and in turn on all of $X_p$, given by:
	\begin{gather*}
 (\P^1)^{p+1}\times \P_{\Sigma^{\vec{d}-\sum_{k=0}^{p} 1_{E_{i_k}}}} \to \P_{\Sigma^{\vec{d}}}\\ ([a_0:b_0], \ldots, [a_p:b_p]), (s'_1,\ldots, s'_n)\mapsto (s_1,\ldots, s_n) \\ \text{such that }\begin{cases*}
		s_j=s'_ j\Pi_{0\leq l \leq p}(a_ly-b_lx) \text{ whenever } \rho_j\in E_{i_k} \text{ for some }k; \\s_j=s'_j \text{ otherwise.}
	\end{cases*}
\end{gather*}
Noting that $X_p$ has $(p+1)^t$ components, indexed by collections of $(p+1)$ not-necessarily-distinct primitive collections $E_{i_0},\ldots, E_{i_p}$, we that that there are natural maps (face maps) $X_p\to X_{p-1}$ given, on each components, by 
	\begin{gather*}\text{ for each } 0\leq l\leq p:\\
	f_l: (\P^1)^{p+1}\times \P_{\Sigma^{\vec{d}-\sum_{k=0}^{p}1_{E_{i_k}}}} \to (\P^1)^{p}\times \P_{\Sigma^{\vec{d}-\sum_{k=0}^{p}1_{E_{i_k}}+ 1_{E_{i_l}}}}\\ ([a_0:b_0], \ldots, [a_p:b_p]), (s'_1,\ldots, s'_n)\mapsto ([a_0:b_0], \ldots, \widehat{[a_l:b_l]}, \ldots, [a_p:b_p])(s_1,\ldots, s_n) \\ \text{such that }\begin{cases*}
		s_j=s'_ j(a_ly-b_lx) \text{ whenever } \rho_j\in E_{i_l}; \\s_j=s'_j \text{ otherwise.}
	\end{cases*}
\end{gather*}

\vspace{4mm}

\textsc{Step 3: }\textit{The corresponding spectral sequence.}

\noindent This semi-simplicial space is actually a $\Delta_{\mathrm{inj}}S$ object in the sense of Fiederowicz-Loday (see \cite{FL}) and we use \cite[Theorem 2]{Banerjee21} on resolution of sheaves of $\Delta_{\mathrm{inj}}S$ objects to compute our desired cohomology.

If $j: \pmord\hookrightarrow {\PSig}^{\vec{d}}$ denotes the open embedding then we get a resolution of $\ell$-adic sheaves or constructible sheaves of $\Qb$-vector spaces (as the case might be) which reads as (see \cite[Lemma 2.11]{Banerjee21}):

\begin{gather}\label{resolSheaves}
j_! \Q_{\pmord}\hookrightarrow\Q_{\P_{\Sigma^{\vec{d}}}}\to \pi_{0_*} \Q_{X_0}\to \big(\pi_{1_*}\Q_{X_1}\otimes \sgn_{S_2}\big)^{S_2}\to \nonumber \\\ldots \to \pi_{p_*}\Q_{X_p}\otimes \sgn_{S_{p+1}}\big)^{S_{p+1}}\to \ldots
\end{gather}

\noindent Applying the functor $\mathrm{R\Gamma}({\textunderscore},\Q_{\P_{\Sigma^{\vec{d}}}})$ and taking cohomology we get a spectral sequence whose $E_1$ page reads as: \begin{gather*}
		E_1^{-p,q} = \begin{cases}
			H^q(\P_{\Sigma^{\vec{d}}})(0) & p=0,\\ \oplus_{1\leq k\leq t} H^{q-2(\epsilon_k-1)}(\P^1\times \P_{\Sigma^{\vec{d}-1_{E_k}}})(-1) & p=1,\\ \oplus_{1\leq i,j \leq t}  H^0(\P^1)\otimes H^2(\P^1)\otimes H^{q-2(\epsilon_k+\epsilon_j-2)-2}(\P_{\Sigma^{\vec{d}-1_{E_k} -1_{E_j}}})(-2) & p=2,\\ 0 & \text{ otherwise},
		\end{cases}
	\end{gather*} with the differentials given by the alternating sum of the Gysin pushforwards induced by the face maps, which is what we shall compute now. 
	\begin{itemize}
		\item \textit{Computing $d_1^{1,q}: E_1^{-1,q} \to E_1^{0,q}$.} 
		
		For simplicity we denote the differential by $d_1^1$. For each $k$ let $\iota_k:  \P_{\Sigma^{\vec{d}-1_{E_k}}}\hookrightarrow  \P_{\Sigma^{\vec{d}}}$ denote the inclusion given by adding a basepoint and $$\iota:=\sqcup \iota_k:  X_0= \bigsqcup_k \P_{\Sigma^{\vec{d}-1_{E_k}}}\hookrightarrow  \P_{\Sigma^{\vec{d}}}$$ denote the induced map on $X_0$. Choose generators $\1\in H^0(\P^1)$ and $e\in H^2(\P^1)$. Also, by $D_{l_1},\ldots, D_{l_m}$ we denote the following cohomology class in 
		$\P(\Sigma^{\vec{d}})$: 
		$$D_{l_1}\cdot \ldots\cdot D_{l_m}:= \sum_{0\leq i_j\leq d_l+1}D_{l_1,i_1}\cdot \ldots\cdot D_{l_r,i_r}.$$ With that shorthand notation setup, we claim that: \begin{gather*}
			d_1^1={f_0}_*: \bigoplus_k H^{*-2(\epsilon_k-1)}(\P^1\times  \P({\Sigma^{\vec{d}-1_{E_k}}}) )\to   H^*( \P(\Sigma^{\vec{d}}))\\ 
			\sum_k \1\otimes \iota_k^*\alpha + \sum_k e\otimes \iota_k^*\beta \mapsto \alpha \sum_{k=1}^{t}\Big(\sum_{j=1}^{\epsilon_k} D_{k_1}\cdot \ldots \cdot \widehat{D_{k_j}}\cdot \ldots \cdot D_{k_{\epsilon_k}}\Big)+ \beta \sum_{k=1}^{t} D_{k_1}\cdot\ldots \cdot D_{k_{\epsilon_k}}
		\end{gather*} is a map of $H^*(\P(\Sigma^{\vec{d}}))$-modules, where $\alpha, \beta\in H^*(\P(\Sigma^{\vec{d}}))$. To see this, first note that $$\iota_k^*: H^*(\P(\Sigma^{\vec{d}})) \to H^*( \P({\Sigma^{\vec{d}-1_{E_k}}}))$$ is a surjection; next, the image of the fundamental class $$[\P^1\times \P({\Sigma^{\vec{d}-1_{E_k}}})]\in H^0(\P^1\times \P(\Sigma^{\vec{d}-1_{E_k}}))$$ is the locus of elements in $\PSig$ that has a basepoint i.e. $\Upsilon(\vec{d})$, which is rationally equivalent, and thus cohomologous, to (a multiple of) $$\sum_{j=1}^{\epsilon_k} D_{k_1}\cdot \ldots \cdot \widehat{D_{k_j}}\cdot \ldots\cdot  D_{k_{\epsilon_k}}$$ and finally, for a fixed point $[a:b]\in \P^1$, the locus given by $$\{[s_0:\ldots: s_r]\in \P_{\Sigma^{\vec{d}}}: s_i([a:b])=0\}$$ is rationally equivalent, and in turn cohomologous, to (a multiple of) $D_{k_1}\cdot\ldots\cdot  D_{k_{\epsilon_k}}$. For the sake of simplicity we won't bother ourselves with the scalar multiples, which is fine because we're working over $\Q$.
		
		It is an easy check that Gysin pushforward $d_1^1={f_0}_*$ surjects onto the ideal generated by $$\sum_{j=1}^{\epsilon_k} D_{k_1}\cdot \ldots \cdot \widehat{D_{k_j}}\cdot \ldots \cdot D_{k_{\epsilon_k}}$$ and $D_{k_1}\cdot\ldots\cdot  D_{k_{\epsilon_k}}$ for all $k$, in $H^*(\P_{\Sigma^{\vec{d}}})$ and that the kernel of $d_1^1$ is given by elements of the form $\iota^*(\alpha)(D_i-e)$ for all $\alpha \in H^*(\PSig)$. 
		
	%	The upshot is that on the $E_2$ page, for $p=0$ we have:\begin{gather}\label{eq3.7}
		%	E_2^{0,q} = \begin{cases}
			%	\Q(0) & q=2j, \,\, 0\leq j\leq 2(r-1)\\0 & \text{ otherwise.}
		%	\end{cases}
		% \end{gather} 
		
		\item  \textit{Computing $d_1^{2,q}: E_1^{-2,q} \to E_1^{-1,q}$.} 
		
		For simplicity, we denote the differential by $d_1^2$. Like before, For each $1\leq k,l\leq t$ let $\iota_{k,l}:  \P(\Sigma^{\vec{d}-1_{E_k}-1_{E_l}}) \hookrightarrow  \P(\Sigma^{\vec{d}-1_{E_k}})$ denote the inclusion given by adding a basepoint and $$\iota:=\sqcup \iota_{k,l}:  X_1= \bigsqcup_k \P(\Sigma^{\vec{d}-1_{E_k}})\hookrightarrow  X_1$$ denote the induced map on $X_1$. Then the way we computed ${f_0}_*$ above works verbatim, and we have \begin{gather*}
			{f_0}_*:\bigoplus_{1\leq k,l \leq t}H^0(\P^1)\otimes H^2(\P^1)\otimes H^{*-2(\epsilon_k+\epsilon_l)} ( \P(\Sigma^{\vec{d}-1_{E_k}-1_{E_l}}) ) \to \bigoplus_k H^{*}( \P(\Sigma^{\vec{d}-1_{E_k}}))\\
			\1 \otimes e\otimes \alpha \mapsto e\otimes\alpha\sum_k \Big( \sum_{j=1}^{\epsilon_k} D_{k_1}\cdot \ldots \cdot \widehat{D_{k_j}}\cdot \ldots \cdot D_{k_{\epsilon_k}} \Big)
		\end{gather*} and 
		\begin{gather*}
			{f_1}_*: \bigoplus_{1\leq k,l \leq t} H^0(\P^1)\otimes H^2(\P^1)\otimes  H^2(\P^1)\otimes H^{*-2(\epsilon_k+\epsilon_l)} (\P_{\Sigma^{\vec{d}-1_{E_k}-1_{E_l}}}) \to \oplus_k H^{*}( \P_{\Sigma^{\vec{d}-1_{E_k}}})\\
			\1 \otimes e\otimes \alpha \mapsto \1\otimes\alpha \sum_k D_{k_1}\cdot\ldots\cdot  D_{k_{\epsilon_k}},
		\end{gather*} and from which we obtain $$d_1^2(1\otimes e \otimes \alpha)= 1\otimes \alpha  \sum_k D_{k_1}\cdot\ldots\cdot  D_{k_{\epsilon_k}} - e\otimes \alpha \sum_k \Big( \sum_{j=1}^{\epsilon_k} D_{k_1}\cdot \ldots \cdot \widehat{D_{k_j}}\cdot \ldots \cdot D_{k_{\epsilon_k}} \Big).$$ Note that $d_1^2$ is injective, and the image is generated by  $$\sum_{j=1}^{c_k} D_{k_1}\cdot \ldots \cdot \widehat{D_{k_j}}\cdot \ldots \cdot D_{k_{\epsilon_k}}$$ and $D_{k_1}\cdot\ldots\cdot  D_{k_{\epsilon_k}}$ for all $k$, in $H^*(X_0)$. 
	%Consequently, on the $E_2$ page we have: \begin{flalign*}
		%	&E_2^{-1,q} = \begin{cases}
			%	\Q(-r) & p=1, q=2j+2r+2, \,\, 0\leq j\leq 2(r-1)\\0 &\text{ otherwise}
			%\end{cases},
			%\\ & E_2^{-2,q} = 0,  \text{ for all } q.
	%	\end{flalign*} 
	\end{itemize}
	In effect on the $E_2$ page all differentials vanish; the spectral sequence degenerates and we obtain the cohomology result of Theorem \ref{thm:PPSig}.
\end{proof}

\section{Higher genus curves on toric varieties}

In this section we prove Theorem \ref{thm:CPSig}.  Recall that $E_1, \ldots, E_t$ are the primitive collections of rays in $\Sigma$. Whereas by $\mathcal{L}_i$ we mean a line bundle that corresponds to the ray $\rho_i$, we will also use the notation $\mathcal{L}_{\rho}$ to denote the line bundle that corresponds to the ray $\rho$- both notations have their uses. Also recall that a morphism from a smooth projective $C$ of genus $g$ to $\PSig$ corresponds to what we call a \emph{non-degenerate $\Sigma$-collection}, up to an equivalence (which we will see soon). \begin{defn}
	We call a set of line bundles $\{\mathcal{L}_1,\ldots, \mathcal{L}_n\}$ to be \emph{$\Sigma$-line bundles} if there exists isomorphisms $\alpha_{e^j}: \bigotimes_{1\leq j\leq n} \mathcal{L}^{\langle e^j, v_{j}\rangle} \to \mathcal{O}_C$, where $\{e^j\}_{1\leq j\leq r}$ is a $\mathbb{Z}$-basis of $M$. 
	
	We call a set of triples $(\mathcal{L}_j,s_j, \alpha_{e^i})$ a \emph{non-degenerate} $\Sigma$-collection if \begin{enumerate}[i]
		\item line bundles $\mathcal{L}_{i}$ for each $\rho_i \in \Sigma(1)$,
		\item global sections $s_{i}\in H^0(C,\mathcal{L}_{i})$ for each $\rho_i \in \Sigma(1)$,
		\item \emph{compatibility isomorphisms:} $\alpha_{e^i}: \bigotimes_{\rho} \mathcal{L}^{\langle e^i, v_{j}\rangle} \to \mathcal{O}_C$,
		\item satisfying the \emph{nondegeneracy} condition: suppose $s_{j}^*: \mathcal{L}_{j}^{-1} \to \mathcal{O}_C$ denotes the map induced by the section $s_{j}$ for each $\rho_j \in \Sigma(1)$, then there is a surjection of locally free sheaves $$ \bigoplus_{\sigma\in \Sigma_{\mathrm{max}}} \big(\otimes_{\{j:\rho_j\notin \sigma(1)\}} s^*_{j}\big): \bigoplus_{\sigma \in \Sigma_{\mathrm{max}}} \big(\otimes_{\{j:\rho_j\notin \sigma(1)\}} \mathcal{L}_{j}^{-1}\big) \to \mathcal{O}_C.$$
		If the last condition of non-degeneracy is not satisfied we simply call it a \emph{$\Sigma$-collection}.
	\end{enumerate}
\end{defn}
A morphism $C\to\PSig$ of degree $\vec{d}$ corresponds to non-degenerate $\Sigma$-collection up to an equivalence that's given by:  $(\mathcal{L}_j,s_j, \alpha_{e^i})\sim  (\mathcal{L}'_j,s'_j, \alpha'_{e^i})$ if there exists a map of line bundles $f: \mathcal{L}_j\to \mathcal{L}'_j$ for all $i$ taking $s_j$ to $s'_j$ and $\alpha_{e^i}\to \alpha'_{e^i}$ for all $1\leq i\leq r, 1\leq j\leq n$. 
\begin{proof}[Proof of Theorem \ref{thm:CPSig}.]
\textsc{Step1:} \textit{Constructing some auxiliary toric bundles.}

Let us first define the space that would play the role in the higher genus case of what $\P_{\Sigma^{\vec{d}}}$ did for genus $0$.

For any integer $d$, let $$\nu: C\times \pic^d C \to \pic^d C$$ denote the projection to the second factor and let $P(d)$ denote a Poincare bundle of degree $d$ on $C\times \pic^d C$, the latter being isomorphic, as varieties, to the Jacobian of $C$ which we denote by $J(C)$. The moduli space of $\Sigma$-collections $(\mathcal{L}_i, s_i, \alpha_v)_{1\leq i \leq n}$ up to equivalence such that $\deg\, \mathcal{L}_i=d_i$, is naturally a \emph{toric bundle} given by $$\P_{\pic_{\Sigma}^{\vec{d}}(C)}(\oplus_i \nu_*P(d_i),\Sigma^{\vec{d}})$$ on $\pic^{\vec{d}}_{\Sigma}(C)$ which we define below. 
To see this note that elements in the space of $\Sigma$-collections have the following description $\{(\mathcal{L}_j, s_j, \alpha_{e^i})_{1\leq j\leq n, 1\leq i\leq r} \} \subset \Pi_i \nu_*P(d_i)\times \A^r$ where $r$ is the rank of $M$, and the equivalence relations are given by sections $\alpha_{e^i}$ themselves. Consider the subvariety of $\Pi_i \pic^{d_i} C$ given by $$\vartheta: \pic^{\vec{d}}_{\Sigma}(C):= \{(\mathcal{L}_1, \ldots, \mathcal{L}_n): \mathcal{L}_i\in \pic^{d_i}C, \,\,\, \otimes_i \mathcal{L}_i^{\langle v,u_i\rangle} \cong \mathcal{O}_C \text{ for all } v\in M\} \hookrightarrow \Pi_i \pic^{d_i} C.$$ Abusing notation and denoting by $\nu_*P(d_i)$ on $\Pi_i\pic^{d_i}(C)$ the vector bundle $\nu_*P(d_i)\boxtimes \mathcal{O}_{\pic^{d_j}(C)}^{\boxtimes{j\neq i}}$, we pull back the vector bundle $\oplus_i \nu_*P(d_i)$ on $\Pi_i \pic^{d_i}(C)$ along the inclusion, which gives us a vector bundle on  $\pic^{\vec{d}}_{\Sigma}$, which in turn, by definition, is the space of all $\Sigma$-collections up to equivalence (note that $\pic^{\vec{d}}_{\Sigma}(C)$ is smooth and isomorphic as an algebraic variety to $\Pi_{\{1\leq j\leq n-r\}} J(C)$).  
Note that when $d_i\geq 2g-1$ for all $i$ (an assumption we maintain for the rest of the paper),  the vector bundle $\oplus_i\vartheta^*\nu_*P(d_i)$ given by the pullback under the inclusion $\vartheta: \pic^{\vec{d}}_{\Sigma}(C)\hookrightarrow \Pi_i \pic^{d_i} C$, has rank $\sum_{j=1}^{n}(d_j-g+1)$ on $\pic^{\vec{d}}_{\Sigma}(C)$.

With the vector bundle on $\pic^{\vec{d}}_{\Sigma}(C)$ defined, to construct the proposed toric bundle $\P_{\pic_{\Sigma}^{\vec{d}}(C)}(\Pi_i \nu_*P(d_i),\Sigma^{\vec{d}})$ we now need to define an exceptional locus, and a reductive linear subgroup of a torus acting on it.
Define an \emph{exceptional set}, much like in the previous section:
	\[
\begin{tikzcd}
\mathcal{Z}_{\mathrm{Pic}^{\vec{d}}_{\Sigma}(C)}(\Sigma^{\vec{d}}) \arrow[hookrightarrow]{rr} \arrow[swap]{dr} & &\vartheta^*\big(\oplus_i\nu_*P(d_i)\big)\arrow{dl}{} \\[10pt]
	& \mathrm{Pic}^{\vec{d}}_{\Sigma}(C)
\end{tikzcd}
\]
such that $\mathcal{Z}_{\mathrm{Pic}^{\vec{d}}_{\Sigma}(C)}(\Sigma^{\vec{d}})$ is locally defined by the vanishing of the monomials coming from the primitive collections of $\Sigma^{\vec{d}}$, which brings us to the necessity of actually defining $\Sigma^{\vec{d}}$ and we do it now.

\noindent Let $\Sigma^{\vec{d}}$ be a rational polyhedral fan whose rays are given by:
\begin{gather*}
	\Sigma^{\vec{d}}(1): = \{\rho_{1,0}, \ldots, \rho_{1,d_1-g+1}, \\  \hspace{3mm}   \rho_{2,0}, \ldots, \rho_{2,d_2-g+1}, \\ \ldots,\\ \rho_{n,0}, \ldots, \rho_{n,d_n-g+1}\}.
\end{gather*} If $\{\rho_{i_1}, \ldots, \rho_{i_l}\}$ is a primitive collection in $\Sigma^{\vec{d}}(1)$ then 	\begin{gather*}
\{\rho_{i_1,0}, \ldots, \rho_{i_1,d_{i_1}-g+1}, \\  \hspace{3mm}   \rho_{i_2,0}, \ldots, \rho_{i_2,d_{i_2}-g+1}, \\ \ldots,\\ \rho_{i_l,0}, \ldots, \rho_{i_l,d_{i_l}-g+1}\}
\end{gather*} is a primitive collection in $\Sigma^{\vec{d}} (1)$, and furthermore, these are the only primitive collections. We define the exceptional set $\mathcal{Z}_{\mathrm{Pic}^{\vec{d}}_{\Sigma}}(\Sigma^{\vec{d}})$ locally over $\mathrm{Pic}^{\vec{d}}_{\Sigma}(C)$ as that cut out by monomials defined by all the primitive collections of $\Sigma^{\vec{d}}(1)$ (note again that defining the rays and the primitive collections, in turn, equivalently, the maximal cones, defines $\Sigma^{\vec{d}}$ uniquely). 
%Now consider $\R^{\sum_{i=1}^{n} (d_i+1)}$ with the rays above along the axes in the positive directions; we now remove the cones spanned by $\{\rho_{i_1,j_1},\ldots, \rho_{i_l,j_l}\}$ if and only if the corresponding $\{\rho_{i_1}, \ldots, \rho_{i_l}\}$ is a primitive collection in $\Sigma(1)$.  

The reductive algebraic group $\mathrm{Hom}({\mathrm{Pic}}(\PSig), \mathbb{G}_m)$ naturally acts on  $$\vartheta^*\big(\oplus_i\nu_*P(d_i)\big)- Z_{\mathrm{Pic}^{\vec{d}}_{\Sigma}(C)}(\Sigma^{\vec{d}})$$ by virtue of its action on $\A^{n}-Z(\Sigma)$ and the resulting geometric quotient is our desired toric bundle $$\P_{\pic_{\Sigma}^{\vec{d}}(C)}(\oplus_i \vartheta^*\nu_*P(d_i),\Sigma^{\vec{d}}).$$
We finish Step 1 by recording the cohomology of these toric bundles (again, recall that $d_i\geq 2g$), which is, in a sense a result of similar nature as \cite[Theorem 1.7]{BPS} and is essentially an application of the Leray-Hirsch theorem for the cohomology of fibre bundles. For a proof refer to \cite{SU}.
%Following Fulton (see \cite{Fulton}) let $\sigma_1, \ldots, \sigma_m$ be an ordering on the maximal cones in $\Sigma$ and let $\tau_i$ be the intersection of $\sigma_i$ with those cones $\sigma_j$ that satisfy $j>i$. Relabelling if necessary, let $\eta_1, \ldots, \eta_r$ be primitive vectors along the rays of $\sigma_m$ and let $\mu_1,\ldots, \mu_r$ be the dual basis.

 \begin{prop}\label{prop:CohomPSig:vecd}
	Let $\P_{\pic_{\Sigma}^{\vec{d}}(C)}(\oplus_i \vartheta^*\nu_*P(d_i),\Sigma^{\vec{d}})$ be as defined above. Then $$H^*(\P_{\pic^{\vec{d}}(\Sigma)} (\Sigma^{\vec{d}});\Qb) \cong H^*(J(C);\Qb)^{\otimes (n-r)}[\{D_{i,j}: 1\leq i\leq n, 0\leq j\leq a_i\}]/ I $$ where $D_{i,j}$ has cohomological degree $2$ and type $(1,1)$ and $I$ is the ideal generated by all \begin{enumerate}
		\item $D_{i_1,j_1}\cdot\ldots\cdot D_{i_k,j_k}$ for $\{\rho_{i_1}, \ldots, \rho_{i_k}\}$ a primitive collection in $\Sigma(1)$;
		\item $ \sum_{i=1}^{n} \langle e^j ,u_i\rangle D_{i,i'}- c_1(\mathcal{L}_{j})$ for each $1\leq j \leq r$, each  $0\leq i'\leq d_i$ and where $c_1(\mathcal{L}_i)$ denote the first Chern class of the torus invariant line bundles $\mathcal{L}_i$ defined above.
	\end{enumerate} 
\end{prop}
Observe that  $\P_{\mathrm{Pic}^{\vec{d}}C}(\oplus_i\nu_*P(d_i),\Sigma^{\vec{d}})$ is a toric bundle as well, over the base space $\pic^{\vec{d}}C:=\Pi_i\pic^{d_i}C$, corresponding the rational polyhedral fan $\Sigma^{\vec{d}}$ in the vector bundle $\oplus_i\nu_*P(d_i)$.

\vspace{7mm}

\textsc{Step 2: }\textit{Constructing a symmetric semi-simplicial space with auxiliary toric bundles.}

	Now we construct a hypercover over an indexing category $\Delta_{\mathrm{inj}} S$ (see \cite{Banerjee21} for more details) which admits universal cohomological descent. 
	
	To this end, we define spaces $\Tc_0$ and $\mathcal{I}_0$ as certain fibre products.
	First, consider the following commutative diagram:
	
		\adjustbox{scale=0.9,center}{
	\begin{tikzcd}
	\Tc_0 \arrow[rr] \arrow[dr,dashed, swap,"\pi_0"] \arrow[dd,swap] &&
		\bigsqcup_{1\leq j\leq t} C\times  \P_{{{\mathrm{Pic}^{{\vec{d}}-1_{E_j}}}}}(\nu_*P(\vec{d}-1_{E_j}),{\Sigma}^{\vec{d}}) \arrow[dd] \arrow[dr,"\oplus_j A_j"] \\
		 & \P_{\pic_{\Sigma}^{\vec{d}}C}(\oplus_i\vartheta^*\nu_*P(d_i),\Sigma^{\vec{d}}) \arrow[rr] \arrow[dd]&&
	 \P_{\mathrm{Pic}^{\vec{d}}C}(\nu_*P(\vec{d}),\Sigma^{\vec{d}}) \arrow[dd] \\
		\Ic_0 \arrow[rr,] \arrow[dr, dashed, "\oplus_j A_j"] && 	\bigsqcup_{1\leq j\leq t}  C\times {\mathrm{Pic}}^{\vec{d}-1_{E_j}}C  \arrow{dr}{\oplus_j A_j}\\
			& {\mathrm{Pic}}^{\vec{d}}_{\Sigma}C\arrow[hookrightarrow, rr] && \mathrm{Pic}^{\vec{d}}	C	 \end{tikzcd}
		}

\noindent	where ${\mathrm{Pic}^{{\vec{d}}}(C)}$ denotes the space $\Pi_i {\mathrm{Pic}}^{d_i}$, $\nu_*P(\vec{d}):= \oplus_i\nu_*P(d_i)$; $\mathrm{Pic}^{{\vec{d}-1_{E_j}}}C$, for a primitive collection $E_j$, denotes the cartesian product of Picard varieties of degree $d_i$ if $\rho_i\notin E_j$, and $d_i-1$ if $\rho_i\in E_j$ and $\nu_*P(\vec{d}-1_{E_j})$ likewise denotes the corresponding vector bundle of global sections; and the maps above are defined by:

	\begin{gather}\label{eq:otimeslambdavec}
		A_j: C\times {\mathrm{Pic}^{{\vec{d}}-1_{E_j}}}C\to {\mathrm{Pic}^{{\vec{d}}}} C\nonumber \\
		x, (L_1, \ldots, L_n)\mapsto L_1\otimes \mathcal{O}_C(x), \cdots, L_n\otimes \mathcal{O}_C(x)
	\end{gather} which is the map of `adding a point': 

$\Ic_0 $ completes the commutative square in the `lower face' of the cube and we abuse notation and still denote the resulting `adding a point' map by $\oplus_j A_j: \Ic_0\to {\mathrm{Pic}}^{\vec{d}}$; the `upper face' of the cube consists of spaces which are essentially the space of global sections of suitable Poincare bundles, i.e. the vertical arrows all correspond to taking fibrewise global sections over the moduli of line bundles; and $\Tc_0$, quite like $\Ic_0$, is defined to complete the square on the upper face of the cube, and admits natural map to $\Ic_0$ so that each of the side faces are also naturally commutative squares. Whereas the `adding a point' map on the lower face of the cube adds points to line bundles, on the upper level they effectively add `basepoints' to global sections; in other words $$\pi_0: \Tc_0\to \P_{{\mathrm{Pic}}^{\vec{d}}(\Sigma)}(\Sigma^{\vec{d}}) $$ simply captures the notion of adding a basepoint.
	
	We define spaces $\Tc_p$ for all $p\geq 0$ likewise. Consider  the following commutative diagram:
	
	\adjustbox{scale=0.9,center}{%
		\begin{tikzcd}[column sep=0]
			\Tc_p \arrow[rr] \arrow[dr,dashed, swap,"\pi_p"] \arrow[dd,swap] &&
			\bigsqcup_{1\leq j_0,\ldots, j_p\leq t} C^{p+1}\times  \P_{{{\mathrm{Pic}^{{\vec{d}}-\sum_{k=0}^{p} 1_{E_{j_k}}}}}C}(\nu_*P({\vec{d}}-\sum_{k=0}^{p} 1_{E_{j_k}}),\Sigma^{\vec{d}}) \arrow[dd] \arrow[dr, "A"] \\
			& \P_{\pic_{\Sigma}^{\vec{d}}C}(\oplus_i\vartheta^*\nu_*P(d_i),\Sigma^{\vec{d}})    \arrow[rr] \arrow[dd]&&
				 \P_{\mathrm{Pic}^{\vec{d}}C}(\nu_*P(\vec{d}),\Sigma^{\vec{d}})  \arrow[dd,"h"] \\
			\Ic_p \arrow[rr,] \arrow[dr, dashed, "A"] &&  \bigsqcup_{1\leq j_0,\ldots, j_p\leq t} C^{p+1}\times\mathrm{Pic}^{\vec{d}-1_{E_{j_k}}}C \arrow{dr}{A}\\
			&  {\mathrm{Pic}}^{\vec{d}}_{\Sigma}C\arrow[hookrightarrow, rr] && \mathrm{Pic}^{\vec{d}}	C
		\end{tikzcd}
	}
	where $A$, by abuse of notation, now denotes the addition of $(p+1)$ points to a line bundle i.e. it's the adding of points for all $1\leq j_0,\ldots, j_p\leq t$: \begin{gather}
		A:   \bigsqcup_{1\leq j_0,\ldots, j_p\leq t} C^{p+1}\times\mathrm{Pic}^{\vec{d}-1_{E_{j_k}}}C  \to \mathrm{Pic}^{\vec{d}} C	\nonumber \\
	(x_0,\ldots, x_p), (L_0, \ldots, L_N)\mapsto L_0(\sum x_i), \cdots, L_N(\sum x_i);
	\end{gather}
	
	\noindent and moreover observe that we have commutative cubes of the following form for all $p$:
	
	\adjustbox{scale=0.7,center}{%
		\begin{tikzcd}[column sep=0.5]
			\Tc_p \arrow[rr] \arrow[dr, swap,"f_i"] \arrow[dd,swap] &&
		\bigsqcup_{1\leq j_0,\ldots, j_p\leq t} C^{p+1}\times  \P_{{{\mathrm{Pic}^{{\vec{d}}-\sum_{k=0}^{p} 1_{E_{j_k}}}}}C}(\nu_*P({\vec{d}}-\sum_{k=0}^{p} 1_{E_{j_k}}),\Sigma^{\vec{d}})  \arrow[dd] \arrow[dr] \\
			& \Tc_{p-1}  \arrow[rr] \arrow[dd]&&
			\bigsqcup_{1\leq j_0,\ldots, j_{p-1}\leq t} C^{p}\times  \P_{{{\mathrm{Pic}^{{\vec{d}}-\sum_{k=0}^{p-1} 1_{E_{j_k}}}}}C}(\nu_*P({\vec{d}}-\sum_{k=0}^{p-1} 1_{E_{j_k}}),\Sigma^{\vec{d}}) \arrow[dd] \\
			\Ic_p \arrow[rr,] \arrow[dr, dashed, "A"] &&  \bigsqcup_{1\leq j_0,\ldots, j_p\leq t} C^{p+1}\times  {{{\mathrm{Pic}^{{\vec{d}}-\sum_{k=0}^{p} 1_{E_{j_k}}}}}}  \arrow{dr}{A}\\
			& \Ic_{p-1}\arrow[rr] && 	 \bigsqcup_{1\leq j_0,\ldots, j_{p-1}\leq t} C^{p}\times  {{{\mathrm{Pic}^{{\vec{d}}-\sum_{k=0}^{p-1} 1_{E_{j_k}}}}}}
		\end{tikzcd}
	} and the all the down-right arrows are maps corresponding to adding a point from the $i^{th}$ factor of $C^{p+1}$, called the \emph{face maps} and we  thus obtain $f_i:\Tc_p\to \Tc_{p-1}$ as the $i^{th}$ face map of the semisimplicial space $\Tc_{\bullet}$. 
	
	It is easy to check that $$\Ic_p \cong \bigsqcup_{1\leq j_0,\ldots, j_p\leq t} C^{p+1}\times (J(C))^{n-r}$$ for all $p\geq 0$. 
	%for $p=0$ the map \begin{gather*}
	%	{\mathrm{Pic}}^{\vec{d}} \times_i \Pi_i {\mathrm{Pic}}^{\lambda_i n}C} C\times \Pi_i %{\mathrm{Pic}}^{\lambda_i n-1}C \to C\times {\mathrm{Pic}}^{n}C\\
		%L, (x,(L_0,\cdots, L_N))\mapsto (x,L)
	%\end{gather*} has a natural inverse thereby giving an isomorphism; for higher values of $p$ the observation follows likewise. 
Similarly it is not hard to show for $p\leq n-2g+1$ that $$\Tc_p\cong  \bigsqcup_{1\leq j_0,\ldots, j_p\leq t} C^{p+1}\times M_{j_0,\ldots, j_p}$$ (the isomorphism is not canonical) for some \emph{toric sub-bundle} $ M_{j_0,\ldots, j_p}$ whose codimension in $\P_{\pic^{\vec{d}}(\Sigma)}(\Sigma^{\vec{d}})$ is $p$, in the following way. 
	For any line bundle $\mathcal{L}$ on $C$ and any effective divisor $D$ on $C$ there is a natural injection \begin{gather*}
		H^0(C,\mathcal{L})\to H^0(C, \mathcal{L}(D))\\ s\mapsto s(D)
	\end{gather*} coming from the short exact sequence of the corresponding locally free sheaves $$0\to \mathcal{L}\to \mathcal{L}(D)\to \mathcal{O}_D\to 0;$$ by definition we have the $\{j_0,\ldots, j_p\}^{th}$ component of $\mathcal{T}_p$ is given by: \begin{gather*}
	\{[s_1:\ldots:s_n], ((x_0,\cdots, x_p),[\widetilde{s_0}:\ldots: \widetilde{s_N}]): s_i\in H^0(C,\mathcal{L}_i) \text{ where } \mathcal{L}_i\in {\mathrm{Pic}}^{d_i},\\ \widetilde{s_i}\in H^0(C,\mathcal{L}'_i) \text{ where } \{\mathcal{L}'_i\}\in {\mathrm{Pic}}^{\vec{d} -\sum_{l=0}^{p} 1_{j_l}},\\ \widetilde{s_i}(\sum x_j)=g(D_i) s_i \text{ for some } g\in \mathrm{Hom}(\pic(\PSig),\mathbb{G}_m)\text{ and for all }i\}.
	\end{gather*} Now fix a general unordered $(p+1)$ subset $\{z_0,\cdots, z_p\}\in C$ and define \begin{gather*}
		M_{j_0,\ldots,j_p}:= \{[\widetilde{s_1}:\ldots: \widetilde{s_n}]: \widetilde{s_i}\in H^0(C,\mathcal{L}_i) \text{ where } \mathcal{L}_i\in \pic^{d_i} C\text{ there exists } \\ g\in \mathrm{Hom}(\pic(\PSig),\mathbb{G}_m) \text{ such that for each }i,\,\,\, \\\widetilde{s_i}(\sum_j z_j)=g(D_i) s_i,\\ \text{ where } s_i\in H^0(C,\mathcal{L}_i), \,\,\,\,  \mathcal{L}\in {\mathrm{Pic}}^{d_i} C\}
	\end{gather*} Note that $M_{j_0,\ldots,j_p}$ is a priori a toric bundle on $\Pi_i \pic^{d_i}C$, which we pullback to $\pic^{\vec{d}}_{\Sigma}(C)$ and still denote the pullback by $M_{j_0,\ldots,j_p}$.
	Then one can define a bijective (easy to check) morphism that only depends on the choice of the points $z_0,\ldots, z_p \in C$: \begin{gather}
	\bigsqcup_{1\leq j_0,\ldots, j_p\leq t}	C^{p+1} \times M_{j_0,\ldots,j_p}\to \Tc_p \nonumber \\ (x_0,\ldots, x_p), [\widetilde{s_1}:\ldots: \widetilde{s_n}]\mapsto\nonumber\\ [\widetilde{s_1}(\sum_j x_j):\ldots: \widetilde{s_n}(\sum_j x_j)], ((x_0,\cdots, x_p), [\widetilde{s_1}:\ldots: \widetilde{s_n}])
	\end{gather} where the only non-trivial part is to check that the map is well-defined and that fact follows from the observation that degree $0$ line bundles always have $\lambda^{th}$ roots for all positive integers $\lambda$. 
Let  $$M_p:= 	\bigsqcup_{1\leq j_0,\ldots, j_p\leq t}  M_{j_0,\ldots,j_p}$$ which is, again, naturally a disjoint union of toric bundles on $\pic^{\vec{d}}_{\Sigma}C$, with fibres isomorphic to a complete simplicial toric variety. Henceforth we denote by $\oplus_i \nu_*P(d_i)$ by $\nu_* P(\vec{d})$, and $\nu_* P(\vec{d}-\sum_j1_{E_j})$ are likewise defined as its sub-vector bundles. 
%$\oplus \nu_* P(d_i)_{z_0+\cdots+z_p}(\Sigma)$. 
In other words, by our notations, $$M_p\cong \P_{\pic^{\vec{d}}_{\Sigma}(C)}\big(\oplus_{1\leq j_0,\ldots, j_p\leq t}\vartheta^*\nu_* P\big(\vec{d}-\sum_{l=0}^{p} 1_{E_{j_l}}\big),\Sigma^{\vec{d}}\big),$$ and for all $0\leq p\leq n-2g+1$ we abuse notation and denote the fibre bundle by $\varphi$  for all $p$ i.e. $$\varphi: M_p\to \pic^{\vec{d}}_{\Sigma}(C).$$
	
	%Define  \begin{align}
		%   T_0 \coloneqq  \{x_0, [s_0:\cdots: s_N]: x_0\in C, s_i\in H^0(C, L^{\otimes\lambda_i}) \text{ for all } i=0, \ldots, N, \,\,\, L\in {\mathrm{Pic}}^n C \nonumber \\s_0, \cdots, s_N \text{ have a common zero at }x_0 \}
		%  \end{align}
	%and let  \begin{align}
		% T_p \coloneqq  \{(x_0,\ldots, x_p), [s_0:\cdots: s_N]: x_j\in C \text{ for all %}j=0, \ldots, p,\nonumber \\ s_i\in H^0(C, L^{\otimes\lambda_i}) \text{ for all }i=0, \ldots, N,\,\, L\in {\mathrm{Pic}}^n C \nonumber \\s_0, \cdots, s_N \text{ have common zeroes at }x_0, \ldots, x_p \}
		% \end{align} i.e. $$T_p = \underbrace{T_0\times_{\mathrm{\widetilde{Hom}}_n (C,\P(\vec{d}))} \times \cdots\times_{\mathrm{\widetilde{Hom}}_n (C,\P(\vec{d}))} T_0}_{p+1}.$$ 
	
	%That $T_{\bullet}\to \mathcal{Z}$ is an augmented simplicial space admitting universal cohomological descent follows immediately from \cite{Deligne}.

\vspace{7mm}

\textsc{Step 3: }\textit{Analysing the resulting spectral sequence.}	Deviating from the norm, we set $\Tc_{-1} \coloneqq \P_{\pic^{\vec{d}}(\Sigma)}(\oplus_i \vartheta^* \nu_*P(d_i), \Sigma^{\vec{d}})$.
	From \cite[Definition 2.1]{Banerjee21}, $\Tc_{\bullet}$ satisfies the conditions of being a $\Delta S$ object. We therefore use \cite[Theorem 1.2]{Banerjee21} to get a second quadrant spectral sequence that reads as: 
	\begin{gather*}
		E_1^{-p,q}= \oplus_i H^{q-2p(\epsilon_i-1)}\big(\Tc_p\otimes \mathit{sgn}_{p+1}\big)^{{S}_{(p+1)}}(-(p+1)(\epsilon_i-1)) \implies H^{q+p}(\cmord).
	\end{gather*}
	where $\mathit{sgn}_{p+1}$ denote the sign action of the symmetric group $S_{p+1}$ on the $(p+1)$ factors of $\Tc_p$ by permutation; $\epsilon_i$ is the cardinality of the primitive collection $E_i$ and the differentials of this spectral sequence are given by the alternating sum of the Gysin pushforwards induced by the face maps. So we split the rest of the proof into the following parts:
	\begin{enumerate}[i.]
		\item \textit{Computing the $E_1$ terms.} To this end note that a complete understanding of $$H^*(\P_{\pic^{\vec{d}}_{\Sigma}(C)}\big(\oplus_{1\leq j_0,\ldots, j_p\leq t}\vartheta^*\nu_* P\big(\vec{d}-\sum_{l=0}^{p} 1_{E_{j_l}}\big),\Sigma^{\vec{d}}\big))$$ (at least in a stable range) gives us full knowledge of the $E_1$ terms (in that range).  The Chern classes of $\nu_* P(d_i)$ can be computed for example, directly using Grothendieck-Riemann-Roch, or via ad-hoc methods: $$c_i(\nu_* P(d_j)) = (-1)^i {\theta^i \over i!}\,\,\, i=0, \ldots, g$$ for all $j$,
		where $\theta$ is the fundamental class of the theta divisor (several proofs are available in \cite[Sections 4, 5, Chapter VII and Section 1, Chapter VIII]{ACGH}). Using the Whitney sum formula and the functoriality of Chern classes under pullbacks, we can express the Chern classes of $\oplus_i \nu_* P(d_i)_{z_0+\cdots+z_p}$ entirely in terms of $\theta$.
		
		%we obtain the Chern classes of $\oplus_i \nu_* P(n)^{\otimes \lambda_i}$: 	\begin{align*}
	%	c_i(\oplus_i \nu_* P(d_i))= \sum_{\substack{0\leq i_0,\ldots, i_k\leq g\\0.i_0+ 1.i_1+2.i_2+\ldots + k i_k =n}} (-1)^{i} {\theta^i \over  i_0!\ldots i_k!} \\ = (-1)^i {\binom{k+i}{i}}{\theta^i\over i!}.
	%	 \end{align*} 
		
		\noindent Let $D_{i,j}$ be the usual generators of $H^*(\P_{\pic^{\vec{d}}_{\Sigma}(C)}\big(\oplus_{1\leq j_0,\ldots, j_p\leq t}\vartheta^*\nu_* P\big(\vec{d}-\sum_{l=0}^{p} 1_{E_{j_l}}\big),\Sigma^{\vec{d}}\big))$ (see Proposition \ref{prop:CohomPSig:vecd}) where $1\leq i\leq n$ and for each $i$, we have $1\leq\ j\leq d_i-g+1$, then $H^*(\P_{\pic^{\vec{d}}_{\Sigma}(C)}\big(\oplus_{1\leq j_0,\ldots, j_p\leq t}\vartheta^*\nu_* P\big(\vec{d}-\sum_{l=0}^{p} 1_{E_{j_l}}\big),\Sigma^{\vec{d}}\big))$, which is an algebra on $H^*(\pic^{\vec{d}}_{\Sigma}C)\cong\big(\wedge(H^1(C))\big)^{\otimes (n-r)}$, is given by Proposition \ref{prop:CohomPSig:vecd}.
		
		%\begin{gather}\label{eq3.9}
			%H^*(\P_{\pic^{\vec{d}}(\Sigma)}(\oplus_i \nu_* P(n)^{\otimes \lambda_i}, \vec{d})\nonumber \\\cong \nonumber \\{H^*({\mathrm{Pic}}^n C)[h]\over {h^{N_0}+\rho^*c^{\vec{\eta}}_1(\oplus_i \nu_* P(n)^{\otimes \lambda_i},\vec{d}) h^{N_0-1}+ \ldots + \rho^*c^{\vec{\eta}}_g(\oplus_i \nu_* P(n)^{\otimes \lambda_i},\vec{d})h^{N_0-g}}}.
	%	\end{gather}

		Let $p$ be such that $\min d_i-p\geq 2g$. Then we have a complete description of the $E_1$ terms of the second quadrant spectral sequence above up to the $-p^{th}$ column because $\oplus_{1\leq j_0,\ldots, j_p\leq t}\vartheta^*\nu_* P\big(\vec{d}-\sum_{l=0}^{p} 1_{E_{j_l}}\big)$ is a toric bundle on $\pic^{\vec{d}}_{\Sigma}C$.
		
		\item \textit{Computing the differentials $d_1^p: E_1^{-p,*}\to E_1^{-(p-1),*+2N}$.}
		
		Following previously introduced notations, let $D_1,
	\ldots, D_n$ denote the usual generators of cohomology of $\PSig$, the corresponding torus invariant divisors on the torus bundle $\P_{\pic^{\vec{d}}_{\Sigma}(C)}\big(\vartheta^*\nu_* P(\vec{d}),\Sigma^{\vec{d}}\big)$ being $D_{i,j}, 1\leq i\leq n, 0\leq j\leq d_i-g+1$ and for all $p$ satisfying $n-p\geq 2g$, let \begin{gather*}
		\iota:\P_{\pic^{\vec{d}}_{\Sigma}(C)}\big(\oplus_{1\leq j_0,\ldots, j_p\leq t}\vartheta^*\nu_* P\big(\vec{d}-\sum_{l=0}^{p} 1_{E_{j_l}}\big),\Sigma^{\vec{d}}\big) \to\\ \P_{\pic^{\vec{d}}_{\Sigma}(C)}\big(\oplus_{1\leq j_0,\ldots, j_{p-1}\leq t} \vartheta^*\nu_* P\big(\vec{d}-\sum_{l=0}^{p-1} 1_{E_{j_l}}\big),\Sigma^{\vec{d}}\big)
	\end{gather*}denote the closed embedding induced by adding a basepoint. Finally, let $e\in H^2(C)$ be the class of a point, $\mathbb{1}$ the fundamental class of $C$, and let $\gamma_1,\ldots, \gamma_{2g}$ be the standard basis of $H^1(C)$ and because $H^*(J(C)) \cong \wedge H^1(C)$, let $\overline{\gamma_1},\ldots, \overline{\gamma_{2g}}$ be the image of $\gamma_1,\ldots, \gamma_{2g}$ under the aforementioned isomorphism. 
		
		First, we observe that \begin{flalign*}
			d_1^{1}: \oplus_{1\leq j\leq t}H^{*-2(\epsilon_j-1)}(C\times \P_{\pic^{\vec{d}}_{\Sigma}(C)}\big(\vartheta^*\nu_* P\big(\vec{d}-1_{E_{j}}\big),\Sigma^{\vec{d}}\big)) \to H^*(\P_{\pic^{\vec{d}}_{\Sigma}(C)}\big(\vartheta^*\nu_* P(\vec{d}),\Sigma^{\vec{d}}\big))\\ \text{ given by}\\
			[C\times\P_{\pic^{\vec{d}}_{\Sigma}(C)}\big(\oplus_{1\leq j\leq t}\vartheta^*\nu_* P\big(\vec{d}-1_{E_{j}}\big),\Sigma^{\vec{d}}\big)] \mapsto \sum_{k=1}^{t}\Big(\sum_{j=1}^{\epsilon_k} D_{k_1}\cdot \ldots \cdot \widehat{D_{k_j}}\cdot \ldots \cdot D_{k_{\epsilon_k}}\Big)
			\\ e\mapsto  \sum_{k=1}^{t} D_{k_1}\cdot\ldots \cdot D_{k_{\epsilon_k}}
			\\ \gamma_i\mapsto \sum_{k=1}^{t} \Big(\sum_{\rho_i\in E_k}\overline{\gamma_i}\Big) \Big(\sum_{j=1}^{\epsilon_k} D_{k_1}\cdot \ldots \cdot \widehat{D_{k_j}}\cdot \ldots \cdot D_{k_{\epsilon_k}}\Big),  \text{ for all }i.
		\end{flalign*} is a map of $H^*(\P_{\pic^{\vec{d}}_{\Sigma}(C)}\big(\vartheta^*\nu_* P(\vec{d}),\Sigma^{\vec{d}}\big))$-modules, and in turn \begin{align*}
\iota^*\alpha + e \iota^*\beta + \sum_{i=1}^{2g} \gamma_i \iota^* \gamma_i \xmapsto{d_1^1} \alpha \sum_{k=1}^{t}\Big(\sum_{j=1}^{\epsilon_k} D_{k_1}\cdot \ldots \cdot \widehat{D_{k_j}}\cdot \ldots \cdot D_{k_{\epsilon_k}}\Big)\\+ \beta \sum_{k=1}^{t} D_{k_1}\cdot\ldots \cdot D_{k_{\epsilon_k}}+ \\ \sum_{i=1}^{2g}  \sum_{k=1}^{t} \Big(\sum_{\rho_i\in E_k}\gamma_i\overline{\gamma_i}\Big) \Big(\sum_{j=1}^{\epsilon_k} D_{k_1}\cdot \ldots \cdot \widehat{D_{k_j}}\cdot \ldots \cdot D_{k_{\epsilon_k}}\Big),
	\end{align*} where $\alpha, \beta, \gamma_1,\ldots, \gamma_{2g} \in  H^*(\P_{\pic^{\vec{d}}_{\Sigma}(C)}\big(\vartheta^*\nu_* P(\vec{d}),\Sigma^{\vec{d}}\big))$.
		Indeed, the justification for the formula for $d_1^1$ in the previous case of $C=\P^1$ holds almost verbatim here. Clearly $$\iota^*: H^*(\P_{\pic^{\vec{d}}_{\Sigma}(C)}\big(\vartheta^*\nu_* P(\vec{d}),\Sigma^{\vec{d}}\big)) \to H^*(\P_{\pic^{\vec{d}}_{\Sigma}(C)}\big(\oplus_{1\leq j\leq t}\vartheta^*\nu_* P\big(\vec{d}-1_{E_{j}}\big),\Sigma^{\vec{d}}\big)) $$ is a surjection; next, the image under adding a particular basepoint is rationally equivalent, and in turn cohomologous, to (a multiple of) $ \sum_{k=1}^{t} D_{k_1}\cdot\ldots \cdot D_{k_{\epsilon_k}}$, and finally, that the image of the fundamental class $$[C\times\P_{\pic^{\vec{d}}_{\Sigma}(C)}\big(\oplus_{1\leq j\leq t}\vartheta^*\nu_* P\big(\vec{d}-1_{E_{j}}\big),\Sigma^{\vec{d}}\big)] \in H^0(C\times\P_{\pic^{\vec{d}}_{\Sigma}(C)}\big(\oplus_{1\leq j\leq t}\vartheta^*\nu_* P\big(\vec{d}-1_{E_{j}}\big),\Sigma^{\vec{d}}\big))$$is rationally equivalent, and thus cohomologous, to (a multiple of)  $$\sum_{k=1}^{t}\Big(\sum_{j=1}^{\epsilon_k} D_{k_1}\cdot \ldots \cdot \widehat{D_{k_j}}\cdot \ldots \cdot D_{k_{\epsilon_k}}\Big)$$ which is because the Poincaré bundle  is $\nu$-relative ample on the Picard variety.
		Recall that a $P(n)$ is $\nu$-relatively very ample for all $n\geq 2g-1$, which in turn induces a relative embedding of $C\times {\mathrm{Pic}}^nC\xrightarrow{i_n}\P (\nu_* P(n))$ over ${\mathrm{Pic}}^n C$ and we have a natural sequence of maps over ${\mathrm{Pic}}^{\vec{d}}_{\Sigma} C$
		\[
		\begin{tikzcd}
			C\times {\mathrm{Pic}}^{\vec{d}}_{\Sigma} C\arrow{rr}{i_{\vec{d}}} \arrow[swap]{dr}{\nu} & &  \P_{ {\mathrm{Pic}}^{\vec{d}}_{\Sigma} C }(\vartheta^*\nu_* P(\vec{d}),\Sigma^{\vec{d}})\arrow{dl}{} \\[10pt]
			& {\mathrm{Pic}}^{\vec{d}}_{\Sigma} C
		\end{tikzcd}
		\] 
	%	This makes $i_n(C\times {\mathrm{Pic}}^n C)$ in $\P_{{\mathrm{Pic}}^n C} (\oplus_i \nu_* P(n)^{\otimes \lambda_i},\vec{d})$ homologous to (a scalar multiple of) the (relative, over the base ${\mathrm{Pic}}^n C$) Poincaré dual of $h \in H^2(\P_{{\mathrm{Pic}}^n C} (\oplus_i \nu_* P(n)^{\otimes \lambda_i},\vec{d}))$. 
	 In turn, the image of  $[C\times\P_{\pic^{\vec{d}}_{\Sigma}(C)}\big(\oplus_{1\leq j\leq t}\vartheta^*\nu_* P\big(\vec{d}-1_{E_{j}}\big),\Sigma^{\vec{d}}\big)]$ under the Gysin map ${f_0}_*$ is given by \begin{align*}
			{f_0}_*\big([C\times\P_{\pic^{\vec{d}}_{\Sigma}(C)}\big(\oplus_{1\leq j\leq t}\vartheta^*\nu_* P\big(\vec{d}-1_{E_{j}}\big),\Sigma^{\vec{d}}\big)] = \sum_{k=1}^{t} D_{k_1}\cdot\ldots \cdot D_{k_{\epsilon_k}}\frown i_{\vec{d}}(C\times {\mathrm{Pic}}^{\vec{d}}_{\Sigma} C)\\ =\sum_{k=1}^{t}\Big(\sum_{j=1}^{\epsilon_k} D_{k_1}\cdot \ldots \cdot \widehat{D_{k_j}}\cdot \ldots \cdot D_{k_{\epsilon_k}}\Big)
		\end{align*}
		Yet again, for the sake of simplicity we won't bother ourselves with the scalar multiples, which is fine because we're working over $\Q$.
		Noting that $$\overline{\gamma_i}(e-D_j)+ D_j(\gamma_i-\overline{\gamma_i}) = \gamma_i D_j- e \overline{\gamma_i},$$ it is now easy to check that the kernel of $d_1^1$ is given by: \begin{flalign*}
		 \oplus_{1\leq j\leq t}H^{*}(\P_{\pic^{\vec{d}}_{\Sigma}(C)}\big(\vartheta^*\nu_* P\big(\vec{d}-1_{E_{j}}\big),\Sigma^{\vec{d}}\big)) (e-D_j) [2(\epsilon_j-1)]\\ \bigoplus   \\ 	
		  \oplus_{1\leq j\leq t}H^{*}(\P_{\pic^{\vec{d}}_{\Sigma}(C)}\big(\vartheta^*\nu_* P\big(\vec{d}-1_{E_{j}}\big),\Sigma^{\vec{d}}\big)) (\gamma_i-\overline{\gamma_i})[2(\epsilon_j-1)], &&(i= 1, \ldots, 2g)
		\end{flalign*} where $[2(\epsilon_j-1)]$ denotes a shift in the cohomological degree by $2(\epsilon_j-1)$. The cokernel of $d^1_1$, which forms $E_2^{0,*}$ is isomorphic to  $$H^*(J(C))^{\otimes (n-r)}[D_1,\ldots, D_n]\over I$$ where $I$ is the ideal generated by the following elements:
	\begin{itemize}
		\item $\sum_{j=1}^{\epsilon_k} D_{k_1}\cdot \ldots \cdot \widehat{D_{k_j}}\cdot \ldots \cdot D_{k_{\epsilon_k}}$ for each primitive collection $E_k$;
		\item $\sum_{i=1}^{n} \langle e^j,u_i\rangle D_{i}$ for all $1\leq j\leq r$.
	\end{itemize}
Note that the Chern classes from Proposition \ref{prop:CohomPSig:vecd} do not appear in the cokernel because they are non-zero only up to degree $g$.
		
		Now we work out the differential for $p=2$ by computing the Gysin pushfowards by each of the face maps: \begin{gather*}
			{f_0}_*(\mathbb{1}\otimes e) = e \sum_{k=1}^{t}\Big(\sum_{j=1}^{\epsilon_k} D_{k_1}\cdot \ldots \cdot \widehat{D_{k_j}}\cdot \ldots \cdot D_{k_{\epsilon_k}}\Big), \\  
			{f_1}_*(\mathbb{1}\otimes e) =  \sum_{k=1}^{t} D_{k_1}\cdot\ldots \cdot D_{k_{\epsilon_k}} \\
			\implies d_1^2 (\mathbb{1}\otimes e) =   \sum_{k=1}^{t}\Big(\sum_{j=1}^{\epsilon_k} (e-D_{k_j}) D_{k_1}\cdot \ldots \cdot \widehat{D_{k_j}}\cdot \ldots \cdot D_{k_{\epsilon_k}}\Big); \\     
			{f_0}_*(e\otimes \gamma_i) = \gamma_i  \sum_{k=1}^{t} D_{k_1}\cdot\ldots \cdot D_{k_{\epsilon_k}}  \\    
			{f_1}_*(e\otimes \gamma_i) = e \overline{\gamma_i}\sum_{k=1}^{t}\Big(\sum_{j=1}^{\epsilon_k} D_{k_1}\cdot \ldots \cdot \widehat{D_{k_j}}\cdot \ldots \cdot D_{k_{\epsilon_k}}\Big) \\ 
			\implies d^2_1(e\otimes \gamma_i)=  \sum_{k=1}^{t}\Big(\sum_{j=1}^{\epsilon_k} (\gamma_i D_{k_j}- e\overline{\gamma_i})D_{k_1}\cdot \ldots \cdot \widehat{D_{k_j}}\cdot \ldots \cdot D_{k_{\epsilon_k}}\Big);\\
			{f_0}_*(\mathbb{1}\otimes \gamma_i) = \gamma_i  \sum_{k=1}^{t}\Big(\sum_{j=1}^{\epsilon_k} D_{k_1}\cdot \ldots \cdot \widehat{D_{k_j}}\cdot \ldots \cdot D_{k_{\epsilon_k}}\Big) \\    {f_1}_*(\mathbb{1}\otimes \gamma_i) = \overline{\gamma_i} \sum_{k=1}^{t}\Big(\sum_{j=1}^{\epsilon_k} D_{k_1}\cdot \ldots \cdot \widehat{D_{k_j}}\cdot \ldots \cdot D_{k_{\epsilon_k}}\Big) \\ 
			\implies d^2_1(\mathbb{1}\otimes \gamma_i) =  (\gamma_i -  \overline{\gamma_i}) \sum_{k=1}^{t}\Big(\sum_{j=1}^{\epsilon_k} D_{k_1}\cdot \ldots \cdot \widehat{D_{k_j}}\cdot \ldots \cdot D_{k_{\epsilon_k}}\Big); \\  \text{ and }d_1^2(\gamma_i\gamma_j)=0,
		\end{gather*} where the last equality follows form the fact that on $\Sym^p H^1(C)$ for $p\geq 2$, the alternating sum of face maps is, by definition, $0$.
Next we see that the $E_2^{-1,*}$ terms, as an $\oplus_{1\leq j,l\leq t}H^{*-2(\epsilon_j+\epsilon_l-2)}(C\times \P_{\pic^{\vec{d}}_{\Sigma}(C)}\big(\vartheta^*\nu_* P\big(\vec{d}-1_{E_{j}-1_{E_l}}\big),\Sigma^{\vec{d}}\big)) )$-module, are given by:
		\begin{align*}
		\oplus_{1\leq j\leq t}	{H^*(J(C)^{\times (n-r)};\Q(-(\epsilon_j-1)))[D_1,\ldots, D_n]\over I}(e-D_j)[2(\epsilon_j)] \\ 
		\bigoplus_{1\leq i\leq 2g} 		\oplus_{1\leq j\leq t} {H^*(J(C)^{\times (n-r)};\Q(-(\epsilon_j-1)))[D_1,\ldots, D_n]\over I}(\gamma_i-\overline{\gamma_i})[2(\epsilon_j)] 
		\end{align*} where $I$ is an ideal generated by elements of the form	\begin{itemize}
		\item $\sum_{j=1}^{\epsilon_k} D_{k_1}\cdot \ldots \cdot \widehat{D_{k_j}}\cdot \ldots \cdot D_{k_{\epsilon_k}}$ for each primitive collection $E_k$;
		\item $\sum_{i=1}^{n} \langle e^j,u_i\rangle D_{i}$ for all $1\leq j\leq r$.
	\end{itemize}
	
	Whereas the kernel of $d_1^2$ is generated by exactly what one expects: as a $\oplus_{1\leq j,l\leq t}H^{*-2(\epsilon_j+\epsilon_l-2)}(C\times \P_{\pic^{\vec{d}}_{\Sigma}(C)}\big(\vartheta^*\nu_* P\big(\vec{d}-1_{E_{j}-1_{E_l}}\big),\Sigma^{\vec{d}}\big)) )$-module, we have  \begin{align*}
			\mathrm{Ker}(d_1^2)= \Bigg( \bigoplus_{1\leq i\leq 2g} \oplus_{1\leq j,l\leq t}H^{*-2(\epsilon_j+\epsilon_l-2)}(C\times \P_{\pic^{\vec{d}}_{\Sigma}(C)}\big(\vartheta^*\nu_* P\big(\vec{d}-1_{E_{j}-1_{E_l}}\big),\Sigma^{\vec{d}}\big))\cdot\\ \big(e\otimes \gamma_i  - 1\otimes \gamma_i D_j +\mathbb{1}\otimes e \overline{\gamma_i}\big)[2(\epsilon_j-1)+2(\epsilon_l-1)]\Bigg)\\ \Bigg(\bigoplus_{1\leq i,j\leq 2g}  \oplus_{1\leq j,l\leq t}H^{*-2(\epsilon_j+\epsilon_l-2)}(C\times \P_{\pic^{\vec{d}}_{\Sigma}(C)}\big(\vartheta^*\nu_* P\big(\vec{d}-1_{E_{j}-1_{E_l}}\big),\Sigma^{\vec{d}}\big))\cdot\\ (\gamma_i\gamma_j)[2(\epsilon_j-1)+2(\epsilon_l-1)]\Bigg).
		\end{align*}
		For $p= 3$ we have $d_1^{3}: E_1^{-3,*}\to E_1^{-2,*}$ given by:
		\begin{gather*}
			d_1^3 (\mathbb{1}\otimes e \otimes \gamma_i) =   \\
			e\otimes \gamma_i  \sum_{k=1}^{t}\Big(\sum_{j=1}^{\epsilon_k} D_{k_1}\cdot \ldots \cdot \widehat{D_{k_j}}\cdot \ldots \cdot D_{k_{\epsilon_k}}\Big)- \\ \mathbb{1}\otimes \gamma_i  \sum_{k=1}^{t} D_{k_1}\cdot\ldots \cdot D_{k_{\epsilon_k}}  +\\
			\mathbb{1}\otimes e \overline{\gamma_i} \sum_{k=1}^{t}\Big(\sum_{j=1}^{\epsilon_k} D_{k_1}\cdot \ldots \cdot \widehat{D_{k_j}}\cdot \ldots \cdot D_{k_{\epsilon_k}}\Big)\\
			\impliedby \begin{cases}
				{f_0}_*(\mathbb{1}\otimes e \otimes \gamma_i) = e\otimes \gamma_i \sum_{k=1}^{t}\Big(\sum_{j=1}^{\epsilon_k} D_{k_1}\cdot \ldots \cdot \widehat{D_{k_j}}\cdot \ldots \cdot D_{k_{\epsilon_k}}\Big),\\ {f_1}_*(\mathbb{1}\otimes e\otimes \gamma_i) = \mathbb{1}\otimes \gamma_i \sum_{k=1}^{t} D_{k_1}\cdot\ldots \cdot D_{k_{\epsilon_k}} \\ {f_2}_*(\mathbb{1}\otimes e\otimes \gamma_i) = \mathbb{1}\otimes e \overline{\gamma_i} \sum_{k=1}^{t}\Big(\sum_{j=1}^{\epsilon_k} D_{k_1}\cdot \ldots \cdot \widehat{D_{k_j}}\cdot \ldots \cdot D_{k_{\epsilon_k}}\Big)
			\end{cases}\\
			d_1^3(e\otimes \gamma_i\gamma_j) = \gamma_i\gamma_j \sum_{k=1}^{t} D_{k_1}\cdot\ldots \cdot D_{k_{\epsilon_k}} , \\
			d_1^3(\mathbb{1}\otimes \gamma_i\gamma_j) = \gamma_i\gamma_j  \sum_{k=1}^{t}\Big(\sum_{j=1}^{\epsilon_k} D_{k_1}\cdot \ldots \cdot \widehat{D_{k_j}}\cdot \ldots \cdot D_{k_{\epsilon_k}}\Big), \\    d_1^3(\gamma_i\gamma_j\gamma_k)=0,
		\end{gather*} where, for the last three equalities, recall again that on $\Sym^p H^1(C)$ for $p\geq 2$, the alternating sum of face maps is, by definition, $0$. Therefore the $E_1^{-2,*}$ terms defined by $\mathrm{Ker}(d^2_1)/ \mathrm{Coker}(d^3_1)$ is given by: 
		\begin{align*}
			\bigoplus_{1\leq i\leq 2g}\oplus_{1\leq j,l\leq t}{H^*({J(C)^{\times (n-r)};\Q(-(\epsilon_j-1)-(\epsilon_l-1))})[D_1,\ldots, D_n]\over I}\cdot\\ \big(e\otimes \gamma_i  - 1\otimes \gamma_i h +\mathbb{1}\otimes e \overline{\gamma_i}\big)[2(\epsilon_j-1)+2(\epsilon_l-1)] \\ \bigoplus_{1\leq i,j\leq 2g}\oplus_{1\leq j,l\leq t}{H^*({J(C)^{\times (n-r)};\Q(-(\epsilon_j-1)-(\epsilon_l-1))})[D_1,\ldots, D_n]\over I}\cdot\\(\gamma_i\gamma_j)[2(\epsilon_j-1)+2(\epsilon_l-1)] .
		\end{align*} where $I$ is the ideal as defined above.
		The formula for the differentials in the case of $p\geq 3$ mimics that of $p=3$, and we have:
		
		\begin{gather*}
			\mathbb{1}\otimes e\otimes \gamma_1\ldots\gamma_{p-2}\mapsto\\  \sum_{k=1}^{t}\Big(\sum_{j=1}^{\epsilon_k} (e\otimes \gamma_1\ldots\gamma_{p-2}-\1\otimes  \gamma_1\ldots\gamma_{p-2} D_{k_j}) D_{k_1}\cdot \ldots \cdot \widehat{D_{k_j}}\cdot \ldots \cdot D_{k_{\epsilon_k}}\Big),\\
			e\otimes \gamma_1\ldots\gamma_{p-1} \mapsto  \gamma_1\ldots\gamma_{p-1}  \sum_{k=1}^{t} D_{k_1}\cdot\ldots \cdot D_{k_{\epsilon_k}},\\  \mathbb{1} \otimes  \gamma_1\ldots\gamma_{p-1}\mapsto \gamma_1\ldots\gamma_{p-1}   \sum_{k=1}^{t}\Big(\sum_{j=1}^{\epsilon_k} D_{k_1}\cdot \ldots \cdot \widehat{D_{k_j}}\cdot \ldots \cdot D_{k_{\epsilon_k}}\Big)\\  \gamma_1\ldots\gamma_{p}\mapsto 0
		\end{gather*}
	and
 \begin{gather*}
			\mathrm{Ker}(d_1^p)/\mathrm{Coker} (d_1^{p+1}) =\\ \Bigg(\bigoplus_{1\leq i\leq 2g}\oplus_{1\leq j_0,\ldots, j_p\leq t}{H^*({J(C)^{\times{(n-r)}};\Q(-\sum_{k=0}^{p}(\epsilon_{j_k}-1))})[D_1,\ldots,D_n]\over I}\cdot\\ \big(e\otimes  \gamma_1\ldots\gamma_{p-1}  - \mathbb{1} \otimes \gamma_1\ldots\gamma_{p-1}  D_j \big)[2\sum_{k=0}^{p}(\epsilon_{j_k}-1)]\Bigg) \\\bigoplus\\ \oplus_{1\leq j_0,\ldots, j_p\leq t}{H^*({J(C)^{\times{(n-r)}};\Q(-\sum_{k=0}^{p}(\epsilon_{j_k}-1))})[D_1,\ldots,D_n]\over I}(\gamma_1\ldots\gamma_{p})[2\sum_{k=0}^{p}(\epsilon_{j_k}-1)].
		\end{gather*}where $I$ is as defined before.
		
	\end{enumerate}
	
	Now we are left with analysing the resulting $E_2$ page. That the differentials on the $E_2$ page vanish for $p\leq \min d_i-2g$ follow simply from weight considerations- the space $\Tc_{\bullet}$ consists of varieties whose $n^{th}$ cohomology is pure of weight $n$. Now observe the following: we have an equality \begin{gather*}
		\mathrm{R}\Gamma_c( \P_{ {\mathrm{Pic}}^{\vec{d}}_{\Sigma} C }(\vartheta^*\nu_* P(\vec{d}),\Sigma^{\vec{d}})),C^{\bullet}(\underline{\Ql}_{\P_{{\mathrm{Pic}}^n C} (\oplus_i \nu_* P(n)^{\otimes \lambda_i},\vec{d})})) =\\ \mathrm{R}\Gamma_c (\P_{ {\mathrm{Pic}}^{\vec{d}}_{\Sigma} C }(\vartheta^*\nu_* P(\vec{d}),\Sigma^{\vec{d}}), j_{!}\underline{\Ql}_{\mathrm{Mor}_n (C,\P(\vec{d}))}) 
	\end{gather*}in the derived category of constructible sheaves over $\P_{ {\mathrm{Pic}}^{\vec{d}}_{\Sigma} C }(\vartheta^*\nu_* P(\vec{d}),\Sigma^{\vec{d}})$ where $C^{\bullet}(\Q_{\P_{ {\mathrm{Pic}}^{\vec{d}}_{\Sigma} C }(\vartheta^*\nu_* P(\vec{d}),\Sigma^{\vec{d}})})$ denotes the complex of $\ell$-adic sheaves: \begin{align}
		0\to j_{!}j^*\underline{\Ql}_{\Tc_{-1}}\to \underline{\Ql}_{\Tc_{-1}}\to \pi_{0_*}\pi_0^*\underline{\Ql}_{\Tc_{-1}} \to (\pi_{1_*}\pi_1^*\underline{\Ql}_{\Tc_{-1}}\otimes \mathit{sgn}_{2})^{{S}_{2}} \cdots \to \nonumber \\ \cdots \to  (\pi_{p_*}\pi_p^* \underline{\Ql}_{\Tc_{-1}}\otimes \mathit{sgn}_{p+1})^{{S}_{p+1}}\to \cdots 
	\end{align} 
	on the other hand, for any $m\in \mathbb{N}$ we have \begin{gather*}
		\mathrm{R}^i\Gamma_c(\P_{ {\mathrm{Pic}}^{\vec{d}}_{\Sigma} C }(\vartheta^*\nu_* P(\vec{d}),\Sigma^{\vec{d}}), C^{\bullet}(\underline{\Ql}_{\P_{{\mathrm{Pic}}^n C} (\oplus_i \nu_* P(n)^{\otimes \lambda_i},\vec{d})}))\cong \\ \mathrm{R}^i\Gamma_c(\P_{ {\mathrm{Pic}}^{\vec{d}}_{\Sigma} C }(\vartheta^*\nu_* P(\vec{d}),\Sigma^{\vec{d}})), C^{\bullet}(\underline{\Ql}_{\P_{ {\mathrm{Pic}}^{\vec{d}}_{\Sigma} C }(\vartheta^*\nu_* P(\vec{d}),\Sigma^{\vec{d}})})/\tau_{\geq m} C^{\bullet}(\underline{\Ql}_{\P_{ {\mathrm{Pic}}^{\vec{d}}_{\Sigma} C }(\vartheta^*\nu_* P(\vec{d}),\Sigma^{\vec{d}})})
	\end{gather*} for all $i\geq 2(m+1)-2\max_j (\epsilon_j-1)$, where $\tau_{\geq m} C^{\bullet}(\Q_{\P_{ {\mathrm{Pic}}^{\vec{d}}_{\Sigma} C }(\vartheta^*\nu_* P(\vec{d}),\Sigma^{\vec{d}})})$ denotes the truncated complex up to the $(\max_j (\epsilon_j-1)-1)$ term and this is because $$\tau_{\geq m} C^{\bullet}(\Q_{\P_{ {\mathrm{Pic}}^{\vec{d}}_{\Sigma} C }(\vartheta^*\nu_* P(\vec{d}),\Sigma^{\vec{d}})})$$ is supported on complex codimension $m$ in $\P_{ {\mathrm{Pic}}^{\vec{d}}_{\Sigma} C }(\vartheta^*\nu_* P(\vec{d}),\Sigma^{\vec{d}})$. Therefore the cohomology of $\mathrm{Mor}_n (C,\P(\vec{d}))$ up to degree $\min d_i-2g$ is solely dictated by the $E_2$ page.
	
	To this end, for each primitive collection $E_j$, let $$z_j:= (e-D_j)$$ which has degree $(-1, 2\epsilon_j)$ and let $$\alpha^j_{i}:= \gamma^j_{i}-\overline{\gamma^j_{i}} \text{ that corresponds to } E_j\,\,\,\,\, i=1,\ldots, 2g$$ which has degree $(-1,2\epsilon_j-1)$. Clearly for $3\leq p\leq n-2g$, the elements $$z_j\alpha^{j_0}_{i_0}\ldots \alpha^{j_p}_{i_p}$$ which are of degree $(-(p+1), 2\epsilon_j+2+\sum_k\epsilon_{j_k})$, when expanded, gives us:  \begin{flalign*}
	z_j\alpha^{j_0}_{i_0}\ldots \alpha^{j_p}_{i_p}\\ &= (e-D_j)(\gamma_{i_1}-\overline{\gamma_{i_1}})\ldots (\gamma_{i_p}-\overline{\gamma_{i_p}}) \\&= (e-D_j)\prod_{j=0}^{p}\gamma_{i_j} +\Big\{ \text{lower order terms as a polynomial on } \gamma_{i_0},\ldots, \gamma_{i_p}\Big\} \\ & = (e-D_j)\prod_{k=0}^{p}\gamma_{i_k}
	\end{flalign*} because the lower order terms are all $0$ in $$\big(H^2(C)\oplus H^0(C)\big)\bigotimes \Sym^p H^1(C)\otimes H^*({\mathrm{Pic}}^{n-(p+1)}C)[h]/h^N,$$ thanks to the alternating action of ${S}_{p+1}$. Whereas  $\alpha^{j_0}_{i_0}\ldots \alpha^{j_p}_{i_p}$, which is of degree $(-(p+1), \sum_k\epsilon_{j_k}$,  when expanded, gives us
	\begin{flalign*}
	\sum_k\epsilon_{j_k}\\ &= (\gamma_{i_0}-\overline{\gamma_{i_0}})\ldots (\gamma_{i_{p}}-\overline{\gamma_{i_{p}}}) \\&= \prod_{j=0}^{p}\gamma_{i_j} +\Big\{ \text{lower order terms as a polynomial on } c_{i_1},\ldots, c_{i_{p+1}}\Big\} \\ & = \prod_{j=1}^{p+1}c_{i_j}
	\end{flalign*}  because again, the lower order terms are all $0$ for the exact same reason cited above.
	
	Now as for $p=2$,  we have \begin{align*}
		z_l\alpha^j_i = (e-D_l)(\gamma_i-\overline{\gamma_i}) = e\gamma_i- \gamma_iD_l+e\overline{\gamma_i}+ D_l\overline{\gamma_i} \\ =e\gamma_i- \gamma_i D_l+e\overline{\gamma_i}
	\end{align*} because the alternating action of ${S_2}$ kills $H^0(C^2)\otimes H^*(\P_{{\mathrm{Pic}}^n C} (\oplus_i \nu_* P(n)^{\otimes \lambda_i},\vec{d}))$, and in turn, $h\overline{\gamma_i}$.
	This give us the algebra structure on the $E_2$ page for $p\leq \min{d_i}-2g$ and thus completes the proof of Theorem \ref{thm:CPSig}.
	
\end{proof}

\section*{Acknowledgements}
I am indebted to Martin Ulirsch for helpful discussions and for bringing this question to my attention. I am also grateful to David Cox and William Fulton for their generous response to my emails.  And my warm thanks to Benson Farb for his helpful feedback and suggestions. This material is based upon work supported by the National Science Foundation under Grant No. DMS-1926686.

\end{document}